\documentclass[12pt,reqno]{amsart}
\usepackage{hyperref}
\usepackage{cleveref}
\usepackage{amsmath,amssymb, amscd, amsfonts, amsthm}
\usepackage{yhmath}
\usepackage[english]{babel}
\usepackage{caption}
\usepackage{euscript,enumerate,tikz}
\usepackage[margin=1in]{geometry}
\usepackage{subcaption}
\usepackage{graphicx}
\usepackage{float}
\usepackage{pgf,tikz}
\usepackage{mathrsfs}
\usepackage{upgreek}
\usepackage[normalem]{ulem}

\def\CLL {{\mathcal L}}
\def\BFC {{\mathbf C}}
\def\BBP {{\mathbb P}}
\def\BBC {{\mathbb C}}
\def\ep {{\epsilon }}

\newcommand\sing{\operatorname{Sing}}

\newtheorem{theorem}{Theorem}[section]
\newtheorem{definition}[theorem]{Definition}

\newtheorem{question}[theorem]{Question}

\theoremstyle{definition}

\newtheorem{example}[theorem]{Example}

\begin{document}
\title[On induced cycles of Levi graphs]{On induced cycles of Levi graphs associated to line arrangements}
\author[Rupam Karmakar]{Rupam Karmakar}
\email{rupammath91@gmail.com}
\address{Stat-Math Unit, Indian Statistical Institute \newline \indent 203 B.T. Road, Kolkata--700108, India.}
\author[Rajib Sarkar]{Rajib Sarkar}
\email{rajib.sarkar63@gmail.com}
\address{Department of Mathematics, National Institute of Technology Warangal \newline \indent Hanamkonda, Telangana-- 506004, India.}

\begin{abstract}
In this article, we investigate the existence of induced cycles in Levi graphs associated to line arrangements in $\mathbb{P}_{\mathbb{C}}^2$. We also look at the problem of finding the length of a longest induced cycle in Levi graphs associated to line arrangements.

\end{abstract}

\keywords{Line Arrangements, Bipartite graphs, Levi graphs, Induced cycles}
\thanks{AMS Subject Classification (2020): 14N10, 14N20, 05C38, 05C10, 05E14}
\maketitle
\section{Introduction}
Over the past several decades the problem of finding a longest induced cycle in a simple graph has garnered much attention in fields such as Combinatorics and Computer science. Erd\H os \cite{Erdos92} mentioned this problem in a more general setting; he talked about the problem of finding a largest induced $r$-regular subgraph in a given graph. Later on, it has been shown that this problem is NP-hard by Cardoso et al. in \cite{CKL07}. Since a cycle is a $2$-regular subgraph, the problem of determining the length of longest induced cycles in a simple graph is NP-hard. Also, it can be shown that finding a longest induced cycle in any simple graph is equivalent to finding a longest induced cycle in any bipartite graph. Let $G$ be a simple graph with the vertex set $V(G)$ and the edge set $E(G)$. We can construct a new bipartite graph $G'$ with the vertex set $V(G')=V(G)\sqcup \{x_{uv} \ | \ \{u,v\}\in E(G)\}$ and the edge set $E(G')=\{\{u,x_{uv}\},\{x_{uv},v\} \ | \ \{u,v\}\in E(G)\}.$ Then there is a path from $u$ to $v$ in $G$ if and only if there is a path from $u$ to $v$ in $G'$. Therefore, a longest induced cycle in $G'$ gives a longest induced cycle in $G$ and vice versa. Several researchers has studied induced cycles in bipartite graphs and it is still an active area of research, see \cite{AA09,Adamus09,Jackson81,Jackson85,LN21} etc.

In \cite{Cox}, Coxeter introduced the notion of Levi graphs that are naturally associated to line arrangements (see Section \ref{Prelim} for the definition). More generally, Levi graphs are naturally associated to plane curve arrangements with ordinary singularities. Levi graphs are important as they encode various properties of intersection posets of curve arrangements with ordinary singularities. It can also be noted that Levi graphs are bipartite graphs. So, the study of finding longest induced cycles in Levi graphs associated to curve arrangements are worth to look at, as this is a significant subclass of bipartite graphs.

In \cite{KSS24}, the authors have studied the binomial edge ideals of Levi graphs associated to certain curve arrangements in the complex projective plane. We prove that the Levi graphs of almost all $d$-arrangements and conic-line arrangements contain an induced cycle of length $6$ and using this property, we obtain bounds for the Castelnuovo-Mumford regularity of powers of binomial edge ideals of Levi graphs coming from $d$-arrangements and conic-line arrangements. Hence, studying the existence of induced cycles in Levi graphs associated to curve arrangements yield lower bounds for the regularity of powers of the corresponding binomial edge ideals. 

Motivated by both these combinatorial and algebraic aspects, we focus our attention to the existence of induced cycles in Levi graphs associated to line arrangements. In \cite[Theorem 5.1]{KSS24}, the authors have proved that for line arrangements $\mathcal{L}$ of $k$ lines with $t_k=0$, the associated Levi graphs has induced cycles of length $6$. In this article, we study the existence of induced cycles of length $\geq 8$ in Levi graphs associated to line arrangements. For line arrangements with $t_k=t_{k-1} =0$, we show that (see Theorem \ref{thm: cycle of length 8}) the associated Levi graphs contain an induced cycle of length $8$. This result is optimal in the sense that there exists line arrangements with $t_k =0$ (see Example \ref{exam: near pencil}) such that the associated Levi graphs does not contain an induced cycles of length $\geq 8$. Now it is natural to ask whether the associated Levi graph of a line arrangement of $k$ lines with $t_k=t_{k-1}= t_{k-2} =0$ contains an induced cycle of length $10$. The answer is negative as we see in Example \ref{exam: not having cycle of length 10}. To study the existence of induced cycles of length $10$, we first consider line arrangements of $k$ lines with largest multiplicity $\geq \frac{k}{2}$ and assuming few additional conditions show the existence of induced cycles of length $10$ (see Theorem \ref{thm: cycles of length 10}). Next, we look into the cases where the largest multiplicity is $ \leq \frac{k}{2}$. We know that only line arrangements with $t_r = 0$ for all $r>2$ are generic line arrangements and in that case the associated Levi graphs contain induced cycles of all possible lengths. So, we consider line arrangements of $k$ lines with $t_3 \neq 0$ and $t_r=0$ for all $r > 3$ and show that the associated Levi graphs contain induced cycles of length $2i$ for $i \leq \lfloor {\frac{k+9}{4}} \rfloor$ if $k$ is odd and for $i \leq  \min \{\lfloor {\frac{k+11}{4}} \rfloor, \lfloor {\frac{2k+16}{7}} \rfloor \}$ if $k$ is even (see Theorem \ref{thm: max t_3}). We generalize this result in some specific settings in Theorem \ref{thm: max t_q} for line arrangements with $t_r \neq 0$ for some $r > 3$.

Geometric properties of a line arrangement fixes the combinatorial data, hence considering a particular class of line arrangement might give better estimate for the lengths of induced cycles in the associated Levi graphs. So, we now restrict ourselves to specific classes of line arrangements, e.g. Ceva's line arrangements and supersolvable line arrangements and study the induced cycles in the associated Levi graphs. However, determining the length of the longest induced cycles even for a particular class of line arrangements seems very difficult. Therefore, we consider the following few examples of line arrangements with varying combinatorial data and obtain the lengths of longest induced cycles for the associated Levi graphs.

$\bullet$ Let $\mathcal{L}_{\{9\}}$ be a $(9)_3$ arrangement in $\mathbb{P}_{\mathbb{C}}^2$. We show that the length of a longest induced cycle in the associated Levi graph is $14$ (see Example \ref{exam: nine lines with t_3}).

$\bullet$ We add a line to the $(9)_3$ arrangement and get a line arrangement $\mathcal{L}_{\{10\}}$ of $10$ lines with $t_2 = 9$ and $t_3 = 12$. We show that the length of a longest induced cycle in the associated Levi graph is $18$ ( see Example \ref{exam: ten lines with t_2 and t_3}).

$\bullet$ Let $\mathcal{L}_{\mathcal{H}}$ be the Hesse arrangement of lines in $\mathbb{P}_{\mathbb{C}}^2$ i.e. $\mathcal{L}_{\mathcal{H}}$ is a line arrangement of $12$ lines with $t_2 =12, t_4 =9$ and $t_r =0$ elsewhere. We show that the length of a longest induced cycle in the associated Levi graph is $12$ ( see Theorem \ref{thm: hesse arrangement}).

$\bullet$ Let $\mathcal{A}$ be a supersolvable line arrangement with two modular points. We compute the length of a longest induced cycle in the associated Levi graph for some cases ( see Theorem \ref{thm: supersolvable k < m-2}).

\section{Acknowledgment}
The first author thanks the National Board for Higher Mathematics (NBHM), Department of Atomic Energy, Government of India for financial support through the postdoctoral fellowship.

\section{Preliminaries}\label{Prelim}
\subsection{Line arrangements and Levi graphs}
   We assume the underlying field to be $\mathbb{C}$, as all the examples used in the article come from complex line arrangements. However, most of the results are true when the underlying field $K$ is aribitrary.
 We recall some basic notation, definitions, and known results that will be used throughout this article.

First, we recall the notion of a line arrangement in the complex projective plane. A \textit{line arrangement} is simply a finite collection of $k$-distinct lines $\{\ell_{1}, \dots , \ell_{k}\}$ in the complex projective plane and is denoted by $\mathcal{L}$. The set of intersection points of $\mathcal{L}$ is denoted by ${\rm Sing}(\mathcal{L})$.


We have the following combinatorial count for a line arrangement $\mathcal{L}$: 
\begin{align}
    \binom{k}{2}= \sum_{p \in \sing(\mathcal{L})} \binom{m_p}{2},
\end{align}
where $m_p$ denotes the multiplicity at $p$ i.e., the number of lines in $\mathcal{L}$ passing through $p \in \sing(\mathcal{L}).$ 
Also, we have for every $\ell_i\in \mathcal{L}$, 
\begin{align}\label{comb count}
    k-1=\sum_{p\in \sing(\mathcal{L})\cap \ell_i}(m_p-1).
\end{align}
Let $t_r(\mathcal{L})$ be the number of $r$-fold points in $\sing{(\mathcal{L})}$ that is the number of points where exactly $r$ lines from $\mathcal{L}$ meet. So, if $s(\mathcal{L})$ is the total number of intersection points in $\mathcal{L}$, then $s(\mathcal{L})= \sum_{r \geq 2}t_r(\mathcal{L})$. We sometimes write $s, t_r$ in place of $s(\mathcal{L}), t_r(\mathcal{L})$ respectively if there is no confusion about $\mathcal{L}$. 

A point $p \in \sing(\mathcal{L})$ is said to be a modular point if for any other point $p' \in \sing(\mathcal{L}) $, the line joining $p$ and $p'$ is $\ell_i \in \mathcal{L}$ for some $i \in \{1, \dots, k\}$.

\vspace{2mm}

We now recall the notion of Levi graphs for line arrangements.
\begin{definition}\label{df: d-arr}
Let $\mathcal{L}=\{\ell_{1}, \dots , \ell_{k}\} \subset \mathbb{P}^{2}_{\mathbb{C}}$ be a line arrangement with $|\sing(\mathcal{L})|=s$. Then the associated \emph{Levi graph} $G = (V,E)$ is a bipartite graph with  $V : = V_{1} \cup V_{2} =$ 
$\{x_{1}, \dots , x_{s}, y_{1}, \dots , y_{k}\}$,  where each vertex $x_i$ corresponds to an intersection point $p_i \in \sing(\mathcal{L})$, each vertex $y_j$ corresponds to the line $\ell_j$ in $\mathcal{L}$ and vertices $x_i$, $y_j$ are joined by an edge in $E$ if and only if $p_i$ is incident with $\ell_j$. 
\end{definition}

\subsection{Some basics from graph theory}
 Let $G$ be a simple graph with the vertex set $V(G)$ and the edge set $E(G)$. A subgraph $G_{sub}$ of $G$ is said to be an \textit{induced subgraph} if for all $u, v \in V(G_{sub})$ such that $\{u,v\} \in E(G)$ implies that $\{u,v\} \in E(G_{sub})$. For a vertex $v$ in $G$,
$N_G(v) := \{u \in V(G) :  \{u,v\} \in E(G)\}$ denotes the
\textit{neighborhood} of $v$ in $G$. The \textit{degree} of a vertex  $v$, denoted by $\deg_G(v)$, is
$|N_G(v)|$. A \textit{cycle} is a connected graph $G$ with $\deg_G(v)=2$ for all $v\in V(G)$.  We denote the cycle on $n$ vertices by $\mathbf{C}_n$ for $n\geq 3$. In particular, $\mathbf{C}_3$ is triangle and $\mathbf{C}_4$ is square. Note that by a cycle we mean it is an induced cycle.


\section{Induced cycle of Levi graphs}\label{Sec: induced cycle}

We have seen in \cite[Theorem 5.1]{KSS24} that for a line arrangement $\mathcal{L}$ of $k$ lines in $\mathbb{P}_{\mathbb{C}}^2$ with $t_k=0$, the associated Levi graph $G$ has an induced $\BFC_{6}$. But in general the associated Levi graphs of line arrangements of $k$ lines with $t_k=0$ doesn't contain induced cycles of length $\geq 8$ as we observe in the following example.

\begin{example} \label{exam: near pencil}
Let $\mathcal{L} = \{ \ell_1, \dots, \ell_k \}$ be a near pencil of $k$ lines i.e., a line arrangement of $k$ lines with $t_2 = k-1, t_{k-1}=1$ and $t_i =0$ elsewhere. Let, $\sing(\mathcal{L}) \cap \ell_k \cap \ell_i$ be double points for $ i = 1, \dots , k-1$ and $\sing(\mathcal{L}) \cap \ell_1 \cap \dots \cap \ell_{k-1}$ be the intersection point of multiplicity $k-1$. Now if possible assume that $\ell_{j_1}p_{j_1j_2}\ell_{j_2} \dots \ell_{j_i}p_{j_ij_1}\ell_{j_1}$ corresponds to an induced cycle of length $2i$ for $i \geq 4$ in the associated Levi graph $G$ of $\mathcal{L}$ where $p_{j_{i_1}j_{i_2}} \in \sing(\mathcal{L}) \cap \ell_{j_{i_1}} \cap \ell_{j_{i_2}}, j_{i_1},j_{i_2} \in \{j_1, \dots, j_i\}$. Since all the double points belong to $\sing(\mathcal{L}) \cap \ell_k$, $\ell_k = \ell_{j_t}$ for some $t \in \{1, \dots, i\}$. Without loss of generality, assume that $\ell_k = \ell_{j_1}$. Then $\{j_2, j_3, j_4\} \subset \{1, \dots, k-1\}$. But that would imply that $p_{j_2j_3} = p_{j_3j_4}$, a contradiction. So, $G$ does not contain an induced cycle of length $\geq 8$.
\end{example}

In this section, we study the existence of induced cycles of length $\geq 8$ in $G$. First we show that, if both $t_{k-1}$ and $t_k$ are zero in $\mathcal{L}$, the associated Levi graph $G$ has an induced $\BFC_8$.
\begin{theorem}\label{thm: cycle of length 8}
    Let $\mathcal{L} = \{\ell_1, \dots \ell_k\} \subseteq \mathbb{P}_{\mathbb{C}}^2$ be a line arrangement with $t_k = t_{k-1} = 0$. Then the associated Levi graph has an induced $\BFC_8$.
\end{theorem}
\begin{proof}
    Let $q$ be the largest integer such that $t_q \neq 0$ i.e., $t_q \neq 0$ and $t_j =0$ for $j > q$. Without loss of generality, assume that $\sing(\mathcal{L}) \cap \ell_1 \cap \dots \cap \ell_q \neq \emptyset$ and denote this point as $p$. Since $t_{k-1} = t_k =0$, we have $q \leq k-2$.

    Let us choose two lines $\ell_{j_1}$ and $\ell_{j_2}$ such that $j_1, j_2 \in \{ q+1, \dots, k\}$ and consider the lines $\ell_1, \ell_{j_1}, \ell_{j_2}$ and $\ell_2$. We denote the intersection points of $\ell_i$ and $\ell_j$ as $p_{ij}$, where both $i$ and $j$ belong to $\{q+1, \dots, k\}$ or one of $i, j$ belongs to $\{1, \dots, q\}$ and the other belongs to $\{q+1, \dots, k\}$.

   If both $\sing(\mathcal{L}) \cap \ell_{j_1} \cap \ell_{j_2} \cap \ell_1$ and $\sing(\mathcal{L}) \cap \ell_{j_1} \cap \ell_{j_2} \cap \ell_2$ are empty, $\ell_1, \ell_{j_1}, \ell_{j_2}, \ell_2$ corresponds to an induced cycle $\BFC_8$ in the following way.

     Since $\sing(\mathcal{L}) \cap \ell_{j_1} \cap \ell_{j_2} \cap \ell_1$ is empty, $\ell_1$ doesn't pass through $\sing(\mathcal{L}) \cap \ell_{j_1} \cap \ell_{j_2}$ and $\ell_{j_2}$ doesn't pass through $\sing(\mathcal{L}) \cap \ell_1 \cap \ell_{j_1}$. Similarly, $\sing(\mathcal{L}) \cap \ell_{j_1} \cap \ell_{j_2} \cap \ell_2 = \emptyset$ implies that $\ell_2$ doesn't pass through $\sing(\mathcal{L}) \cap \ell_{j_1} \cap \ell_{j_2}$ and $\ell_{j_1}$ doesn't pass through $\sing(\mathcal{L}) \cap \ell_2 \cap \ell_{j_2}$. Also by hypothesis, $\ell_{j_1}$ and $\ell_{j_2}$ don't pass through $\sing(\mathcal{L}) \cap \ell_1 \cap \ell_2$. So, $\ell_1p_{1j_1}\ell_{j_1}p_{j_1j_2}\ell_{j_2}p_{j_22}\ell_2p\ell_1$ corresponds to an induced $\BFC_8$ in the associated Levi graph.

       If one of $\sing(\mathcal{L}) \cap \ell_{j_1} \cap \ell_{j_2} \cap \ell_1$ or $\sing(\mathcal{L}) \cap \ell_{j_1} \cap \ell_{j_2} \cap \ell_2$ is non-empty, say $\sing(\mathcal{L}) \cap \ell_{j_1} \cap \ell_{j_2} \cap \ell_2$ is non-empty, we replace $\ell_2$ with $\ell_3$ and $\ell_1, \ell_{j_1}, \ell_{j_2}, \ell_3$ corresponds to an induced $\BFC_8$ in the associated Levi graph. Similarly, if $\sing(\mathcal{L}) \cap \ell_{j_1} \cap \ell_{j_2} \cap \ell_1$ is non-empty, we replace $\ell_1$ with $\ell_3$ and $\ell_3, \ell_{j_1}, \ell_{j_2}, \ell_2$ corresponds to an induced $\BFC_8$ in that case.
     \end{proof}

From \cite[Theorem 5.1]{KSS24} and Theorem \ref{thm: cycle of length 8} one might ask the following question:
 \begin{question}\label{qs: cycle of length 10}
For a line arrangement $\mathcal{L}$ of $k$ lines in $\mathbb{P}_{\mathbb{C}}^2$ with $t_k = t_{k-1} = t_{k-2} = 0$, does the associated Levi graph has an induced $\BFC_{10}$?
\end{question}

But this is not true as we observe in the example below.

\begin{example}\label{exam: not having cycle of length 10}
Let $\mathcal{L}$ be a line arrangement with two modular points $p_1$ and $p_2$ such that $a \geq 5$ lines pass the modular point $p_1$ and $b > a$  lines pass through the other modular point $p_2$. Then $k = a+ b-1 \geq 10$ and $t_2 = (a-1)(b-1), t_a=1, t_b = 1$.

Let, $q$ be the largest integar such that $t_q \neq 0$. Here $q=b$ since $b$ is the largest integer such that $t_b \neq 0$. Also, $t_{k -q+1} = t_a \neq 0$. We claim that the associated Levi graph of $\mathcal{L}$ does not have an induced cycle of dimension $\geq 10$.

Let $\{\ell_{j_1}, \dots, \ell_{j_i}\} \subset \mathcal{L}$ correspond to an induced $\mathbf{C}_{2i}$ for $i \geq 5$ in the associated Levi graph. Then at least three of the lines pass through one of the modular points, say $\ell_{j_{i_1}}, \ell_{j_{i_2}}, \ell_{j_{i_3}}$ pass through $p_1$. Then $\sing(\mathcal{L}) \cap \ell_{j_{i_1}} \cap \ell_{j_{i_2}} \cap \ell_{j_{i_3}} = \{p_1\}$. Hence, $\{\ell_{j_1}, \dots \ell_{j_i}\}$ can not correspond to an induced cycle of length $ 2i \geq 10$.
\end{example}

 \vskip 4mm
\begin{theorem}\label{thm: cycles of length 10}
    Let $\mathcal{L} = \{\ell_1, \dots, \ell_k \}$ be an arrangement of $k \geq 10$ lines in $\mathbb{P}_{\mathbb{C}}^2$ with $q$ the largest integer such that $t_q \neq 0$. Without loss of generality we may assume that $\sing(\mathcal{L}) \cap \ell_1 \cap \dots \cap \ell_q \neq \emptyset$.
    
    \vskip 2mm
    $(i)$ If $t_{k-q} = 0$, $t_{k-q+1} = 0$ and  $k \leq 2q-2$, the associated Levi graph $G$ has an induced $\mathbf{C}_{10}$.

 If $t_{k-q} \neq 0$, the associated Levi graph $G$ has an induced $\BFC_{10}$ if and only if  $\sing(\mathcal{L}) \cap \ell_{{q+1}}\cap \dots \cap \ell_k = \emptyset$.

 $(ii)$ For $k=2q-1$ with $t_{k-q} = 0$, the associated Levi graph $G$ has an induced $\BFC_{10}$ if and only if $t_q = 1$.

 $(iii)$ For $k=2q$, the associated Levi graph $G$ has an induced $\BFC_{10}$ if and only if $t_q=1$.

\end{theorem}
\begin{proof}

\textit{Proof of $(i)$:}
   Let $t_{k-q} = 0$.  Since $t_{k-q} = 0$ and $t_{k-q+1}=0$, we have $\sing(\mathcal{L}) \cap \ell_{q+1} \cap \dots \cap \ell_k = \emptyset$ and $\sing(\mathcal{L}) \cap \ell_* \cap \ell_{q+1} \cap \dots \cap \ell_k  = \emptyset$ for all $ * \in \{1, \dots, q\}$. Then for $q+1$ and $q+2$ there exists $k_1 \in \{q+3, \dots k\}$ such that $\sing(\mathcal{L}) \cap \ell_{q+1} \cap \ell_{q+2} \cap \ell_{k_1} = \emptyset$.
    
    Now consider $\ell_1, \ell_2, \ell_{q+1}, \ell_{q+2}, \ell_{k_1}$.
    
    \textbf{Case I}:
    Let either $\sing(\mathcal{L}) \cap \ell_{q+1} \cap \ell_{q+2} \cap \ell_1$ or $\sing(\mathcal{L}) \cap \ell_{q+1} \cap \ell_{q+2} \cap \ell_2$  be non empty, say $\sing(\mathcal{L}) \cap \ell_{q+1} \cap \ell_{q+2} \cap \ell_1 \neq \emptyset$. then we replace $\ell_1$ by $\ell_3$ in the above diagram. Now consider $\ell_{q+2}$ and $\ell_{k_1}$.

\textbf{Subcase I}: If both of $\sing(\mathcal{L}) \cap \ell_{q+2} \cap \ell_{k_1} \cap \ell_2$ and  $\sing(\mathcal{L}) \cap \ell_{q+2} \cap \ell_{k_1} \cap \ell_3$ are empty then we don't change anything and consider $\ell_{q+1}$ and $\ell_{k_1}$ in the next step. 
\begin{itemize}

\item If both of $\sing(\mathcal{L}) \cap \ell_{q+1} \cap \ell_{k_1} \cap \ell_2$  and $\sing(\mathcal{L}) \cap \ell_{q+1} \cap \ell_{k_1} \cap \ell_3$  are empty then we get an induced $\mathcal{C}_{10}$ corresponding to $\{\ell_3, \ell_2, \ell_{q+1}, \ell_{q+2}, \ell_{k_1}\}$ in the following way.

    Let $\sing(\mathcal{L}) \cap \ell_1 \dots \cap \ell_q  = \{p_1\}, \sing(\mathcal{L}) \cap \ell_{q+1} \cap \ell_3 = p_2, \sing(\mathcal{L}) \cap \ell_{q+1} \cap \ell_{q+2} = p_3,  \sing(\mathcal{L}) \cap \ell_{q+2} \cap \ell_{k_1} = p_4, \sing(\mathcal{L}) \cap \ell_{k_1} \cap \ell_2  = p_5$. Since $\sing(\mathcal{L}) \cap \ell_{q+1} \cap \ell_{q+2} \cap \ell_1  \neq \emptyset$, $p_2 \notin \ell_{q+2}$ and since $\sing(\mathcal{L}) \cap \ell_{q+1} \cap \ell_{k_1} \cap \ell_3  = \emptyset$, $p_2 \notin \ell_{k_1}$. Also $\sing(\mathcal{L}) \cap \ell_3 \cap \ell_2 = p_1$ implies that $p_2 \notin \ell_2$.

    By hypothesis, $p_3 \notin  \ell_{k_1}$ and $p_4 \notin \ell_{q+1}$. furthermore $\sing(\mathcal{L}) \cap \ell_{q+1} \cap \ell_{q+2} \cap \ell_{\ast}= \emptyset$ for $\ast =2, 3$ implies that $p_3 \notin \{ \ell_3, \ell_2 \}$ and $\sing(\mathcal{L}) \cap \ell_{i+2} \cap \ell_{k_1} \cap \ell_{\ast} = \emptyset$ for $\ast =2,3$ implies that $p_4 \notin \{\ell_2, \ell_3 \}$. Similarly one can show that $p_5 \notin \{\ell_3, \ell_{q+1}, \ell_{q+2})$. So, $p_i$ for $i=1, \dots ,5$ along with $\{\ell_3, \ell_2, \ell_{q+1}, \ell_{q+2}, \ell_{k_1}\}$ correspond to an induced $\mathcal{C}_{10}$ in $G$.
\vspace{2mm}

  \item  If one of $\sing(\mathcal{L}) \cap \ell_{q+1} \cap \ell_{k_1} \cap \ell_2$ and $\sing(\mathcal{L}) \cap \ell_{q+1} \cap \ell_{k_1} \cap \ell_3$  is non-empty, say $\sing(\mathcal{L}) \cap \ell_{q+1} \cap \ell_{k_1} \cap \ell_3  \neq \emptyset$, then we replace $\ell_3$ by $\ell_4$. Now if $\sing(\mathcal{L}) \cap \ell_{q+2} \cap \ell_{k_1} \cap \ell_4 = \emptyset$ then we claim that  $\{\ell_4, \ell_2, \ell_{q+1}, \ell_{q+2}, \ell_{k_1}\}$ correspond to an induced $\mathcal{C}_{10}$ in $G$. 

    To prove our claim let $\sing(\mathcal{L}) \cap \ell_{q+1} \cap \ell_4 = p_2^{\ast}$. Now similar arguments as above says that $p_2^{\ast} \notin \{ \ell_{q+2}, \ell_{k_1} \}$ and $p_3, p_4, p_5 \notin \ell_4$. So, $p_1, p_2^{\ast}, p_3, p_4, p_5$ along with $\{\ell_4, \ell_2, \ell_{q+1}, \ell_{q+2}, \ell_{k_1}\}$ correspond to an induced $\mathbf{C}_{10}$ in $G$.

    If $\sing(\mathcal{L}) \cap \ell_{q+2} \cap \ell_{k_1} \cap \ell_4 \neq \emptyset$ then replace $\ell_4$ by $\ell_5$ and in that case
    
    $\{\ell_5, \ell_2, \ell_{q+1}, \ell_{q+2}, \ell_{k_1}\}$ correspond to an induced $\mathcal{C}_{10}$ in $G$.

    \end{itemize}
    
    \textbf{Subcase II}:  One of $\ell_{(q+2)(k_1)2}$ and $\ell_{(q+2)(k_1)3}$ is non-empty, say $\ell_{(q+2)(k_1)3} \neq \emptyset$. Then we replace $\ell_3$ by $\ell_4$ and proceed. 
    \begin{itemize}
        \item  If both of $\sing(\mathcal{L}) \cap \ell_{q+1} \cap \ell_{k_1} \cap \ell_2$  and $\sing(\mathcal{L}) \cap \ell_{q+1} \cap \ell_{k_1} \cap \ell_4$  are empty then we get an induced  $\mathcal{C}_{10}$ corresponding to $\{\ell_4, \ell_2, \ell_{q+1}, \ell_{q+2}, \ell_{k_1}\}$ as shown in \textbf{Subcase I}.

        \item If one of $\sing(\mathcal{L}) \cap \ell_{q+1} \cap \ell_{k_1} \cap \ell_2$ and $\sing(\mathcal{L}) \cap \ell_{q+1} \cap \ell_{k_1} \cap \ell_4$  is non-empty, say $\sing(\mathcal{L}) \cap \ell_{q+1} \cap \ell_{k_1} \cap \ell_4  \neq \emptyset$, then we replace $\ell_4$ by $\ell_5$. If $\sing(\mathcal{L}) \cap \ell_{q+2} \cap \ell_{k_1} \cap \ell_5 = \emptyset$ then we claim that  $\{\ell_5, \ell_2, \ell_{q+1}, \ell_{q+2}, \ell_{k_1}\}$ correspond to an induced $\mathcal{C}_{10}$ in $G$. Otherwise replace $\ell_5$ by $\ell_6$ and in that case $\{\ell_5, \ell_2, \ell_{q+1}, \ell_{q+2}, \ell_{k_1}\}$ correspond to an induced $\mathbf{C}_{10}$.
    \end{itemize}
  
    \textbf{Case II}:
    If Both of  $\sing(\mathcal{L}) \cap \ell_{q+1} \cap \ell_{q+2} \cap \ell_1$ and $\sing(\mathcal{L}) \cap \ell_{q+1} \cap \ell_{q+2} \cap \ell_2$ are empty then we make no changes and proceed. In the next step we consider $\ell_{q+2}$ and $\ell_{k_1}$ and divide the proof into two subcases as in \textbf{Case I}. The proof goes along the lines of \textbf{Case I} and the details are left to the reader.

    Now let $t_{k-q} \neq \emptyset$. We claim that $t_q, t_{k-q} \neq 0$ implies that $t_{k-q-1} = 0$. If possible assume that $t_{k-q-1} \neq 0$. Since $\sing(\mathcal{L}) \cap \ell_1 \cap \dots \cap \ell_q \neq \emptyset$, only possibility of a $(k-q-1)$-fold point is
    $\sing(\mathcal{L}) \cap \ell_{*} \cap \ell_{{q+1}}\cap \dots \cap \ell_k$ for some $* \in \{1, \dots, q\}$. But that would imply that $t_{k-q}=0$, a contradiction.

    Now to prove second part of $(i)$, assume that $\sing(\mathcal{L}) \cap \ell_{{q+1}}\cap \dots \cap \ell_k = \emptyset$. Since $t_{k-q+1} = 0$, $\sing(\mathcal{L}) \cap \ell_* \cap \ell_{{q+1}}\cap \dots \cap \ell_k = \emptyset$ for all $* \in \{1, \dots, q\}$. Now the proof is similar to the previous case. 

    To prove the converse suppose $\sing(\mathcal{L}) \cap \ell_{{q+1}}\cap \dots \cap \ell_k \neq \emptyset$. We will show that $G$ does not have an induced $\BFC_{10}$. Suppose there exists an induced cycle of length ten in $G$ and $\ell_{i_1}p_{i_1i_2}\ell_{i_2} \dots \ell_{i_5}p_{i_5i_1}\ell_{i_1}$ corresponds to an induced cycle of length ten. Then there exists a three element subset $S =  \subset \{ \ell_a, \ell_b, \ell_c \} \subset \{ \ell_{i_1}, \dots, \ell_{i_5} \}$ such that $\{a,b,c \} \subset \{1, \dots, q \}$ or $\{a,b,c \} \subset  \{q+1, \dots, k \}$. 

    If $\{a,b,c \} \subset \{1, \dots, q \}$, $\sing(\mathcal{L}) \cap \ell_a \cap \ell_b \cap \ell_c = \sing(\mathcal{L})\cap \ell_1 \cap \dots \cap \ell_q$, a contradiction.

    Similarly, if $\{a,b,c\} \subset \{q+1, \dots, k \}$, $\sing(\mathcal{L}) \cap \ell_a \cap \ell_b \cap \ell_c = \sing(\mathcal{L})\cap \ell_{q+1} \cap \dots \cap \ell_k$, a contradiction. So, $G$ does not have an induced cycle of length ten.

    \textit{Proof of $(ii)$:}

    For $k=2q-1$, $t_{k-q} =0$ implies that $\sing(\mathcal{L}) \cap \ell_{q+1} \cap \dots \cap \ell_k = \emptyset$. Also $t_q=1$ implies that $\sing(\mathcal{L}) \cap \ell_* \cap \ell_{{q+1}}\cap \dots \cap \ell_k = \emptyset$ for all $* \in \{1, \dots, q\}$. Now, the proof to show that $G$ contains an induced cycle of length ten is similar to the proof of $(i)$. To prove the converse assume that $t_q > 1$. Then $\sing(\mathcal{L}) \cap \ell_i \cap \ell_{{q+1}}\cap \dots \cap \ell_k \neq \emptyset$ for some $i \in \{1, \dots, q\}$. Suppose, $\ell_{i_1}p_{i_1i_2}\ell_{i_2} \dots \ell_{i_5}p_{i_5i_1}\ell_{i_1}$ corresponds to an induced cycle of length ten in $G$. Then we can argue as in $(i)$ to show that there exists a subset $\{a,b,c\} \subset \{i_1, \dots, i_5 \}$ such that $\{a,b,c \} \subset \{1, \dots, q \}$ or $\{a,b,c \} \subset  \{q+1, \dots, k \}$, a contradiction. 

    \textit{Proof of $(iii)$}: we have $t_{k-q+1} =0$ as $q$ is the largest multiplicity of points in $\sing(\mathcal{L})$ and $k-q +1 = q+1 > q$. So, $\sing(\mathcal{L}) \cap \ell_* \cap \ell_{{q+1}}\cap \dots \cap \ell_k = \emptyset$ for all $* \in \{1, \dots, q\}$. Next, $t_q=1$ implies that $\sing(\mathcal{L}) \cap \ell_{q+1} \cap \dots \cap \ell_k = \emptyset$. Now, the proof is similar to the proof of $(ii)$, hence omitted.

\end{proof}

Now we shift our attention to line arrangements of $k$ lines with $k < 2q$, where $q$ is the largest multiplicity of points in $\sing(\mathcal{L})$. We know that only line arrangements with $t_r =0$ for all $r > 2$ are the generic line arrangements and the associated Levi graphs in that case contain cycles of length $2i$ for all $ 2 
\leq i \leq k$. So, we consider line arrangements with $t_r=0$ for all $r > 3$ and study the induced cycles in the associated Levi graphs.

    \vskip 6mm



\vspace{ 3mm}

\begin{theorem}\label{thm: max t_3}
    Let $\mathcal{L} = \{ \ell_1, \dots \ell_k \}$ be a line arrangement of $k$ lines with $t_3 \neq 0$ and $t_r = 0$ for all $r > 3$. Then

    $(i)$If $k$ is odd then the associated Levi graph has an induced $\BFC_{2i}$, for all $i \leq \lfloor {\frac{k+9}{4}} \rfloor$.

    $(ii)$ if $k$ is even and atleast one line has more than one double point in it, the associated Levi graph has an induced $\BFC_{2i}$ for all $i \leq \lfloor {\frac{k+11}{4}} \rfloor$.
    
    $(iii)$ if $k$ is even and every line in $\mathcal{L}$ has exactly one double point in it, the associated Levi graph has an induced $\BFC_{2i}$ for all $i \leq \lfloor {\frac{2k+16}{7}} \rfloor$.
\end{theorem}

\begin{proof}
    Let $\ell_{j_1}, \dots , \ell_{j_i}$ corresponds to an induced cycle $\BFC_{2i}$ in $G$.
    
   \textit{Proof of $(i)$:} We start from $\ell_{j_1}$ and consider $\ell_{j_1}$ and $\ell_{j_2}$ as first two lines. Then we have $k-2$ or $k-3$ choices for the third line $\ell_{j_3}$ in $\mathcal{C}$ depending on whether the intersection point of $\ell_{j_1}$ and $\ell_{j_2}$ is a double or triple point. So, there are at least $k-3$ choices for $\ell_{j_3}$. If both $\sing(\mathcal{L}) \cap \ell_{j_1} \cap \ell_{j_2}$ and $\sing(\mathcal{L})  \cap \ell_{j_2} \cap \ell_{j_3}$ are triple points then we have $k-5$ choices for $\ell_{j_4}$. Also, we have to remove the possibility of $\sing(\mathcal{L}) \cap \ell_{j_1} \cap \ell_{j_3} \cap \ell_{j_4}$ being non-empty. So, there are at least $k-6$ choices for $\ell_{j_4}$.

   Similarly for $5 \leq i \leq \lfloor {\frac{k+5}{3}} \rfloor $, if all of $\sing(\mathcal{L}) \cap \ell_{j_1} \cap \ell_{j_2}, \sing(\mathcal{L}) \cap \ell_{j_2} \cap \ell_{j_3}, \dots , \sing(\mathcal{L}) \cap \ell_{j_{i-1}} \cap \ell_{j_i}$ are triple points then we get at least $k -{(i-1) + (i -2)}$ choices for $\ell_{j_i}$. But it might happen that all of $\sing(\mathcal{L}) \cap \ell_{j_1} \cap \ell_{j_{t-1}}, \dots , \sing(\mathcal{L}) \cap \ell_{j_{i-3}} \cap \ell_{j_i}$ are triple points. In that case, we remove $i-3$ more choices for $\ell_{j_i}$. Also, the intersection point of $\ell_{j_1}$ and $\ell_{j_i}$ might pass through any one of $\ell_{j_t}$ for $t=3, \dots, i-2$. So, for an induced cycle of length $2i$ we have to further remove $i-4$ choices for $\ell_{j_i}$. Hence, we have at least $k - \{ (i-1) +(i-2) +(i-3) +(i-4)\} = k -(4i-10)$ choices for $\ell_{j_i}$. Now, for $i= \lfloor {\frac{k+9}{4}} \rfloor$ there are at least $k - (4 \lfloor {\frac{k+5}{3}} \rfloor -10 \geq 1$ choices for $\BFC_{j_{\lfloor {\frac{k+5}{3}} \rfloor}}$. That proves the existence of an induced cycle of length $2i$ for every $ i \leq \lfloor {\frac{k+9}{4}} \rfloor$.

 Note that if $k$ is even then every line has at least one double point i.e. $\sing(\mathcal{L}) \cap \ell_i$ contain at least one double point for all $i \in \{1, \dots, k\}$. We divide it into two cases.
 
\vspace{2mm}
\textit{Proof of $(ii)$:}
 Without loss of generality, let us assume that $\ell_{j_2}$ has at least two double points in it and they are $\sing(\mathcal{L}) \cap \ell_{j_2} \cap \ell_{j_1}$ and $\sing(\mathcal{L}) \cap \ell_{j_2} \cap \ell_{j_3}$. Then we have $k-3$ or $k-4$ choices for $\ell_{j_4}$ depending on whether the intersection point of $\ell_{j_1}$ and $\ell_{j_3}$ is a double or a triple point. So, we have at least $k-4$ choices for $\ell_{j_4}$. 

 If any one of $\sing(\mathcal{L}) \cap \ell_{j_4} \cap \ell_{j_1}$, $\sing(\mathcal{L}) \cap \ell_{j_4} \cap \ell_{j_2}$ or $\sing(\mathcal{L}) \cap \ell_{j_4} \cap \ell_{j_3}$ is a double point, then we have at least $k -(4+2)$ choices for $\ell_{j_5}$. Otherwise, we can choose $\ell_{j_5}$ such that $\sing(\mathcal{L}) \cap \ell_{j_4} \cap \ell_{j_5}$ is a double point and proceed to the next step.

 Next we look at the choices for  $\ell_{j_6}$. If one of $\sing(\mathcal{L}) \cap \ell_{j_5} \cap \ell_{j_1}$, \dots, $\sing(\mathcal{L}) \cap \ell_{j_5} \cap \ell_{j_4}$ is a double point, then we have at least $k- (5+3)$ choices for $\ell_{j_6}$. But if $\sing(\mathcal{L}) \cap \ell_{j_3} \cap \ell_{j_4}$ is a triple point,  we need to remove one more possibility. So we have at least $k-(5+3+1)$ choices for $\ell_{j_6}$. Otherwise, similar to the previous case we can choose $\ell_{j_6}$ such that $\sing(\mathcal{L}) \cap \ell_{j_5} \cap \ell_{j_6}$ is a double point.

 Proceeding similarly we have at least $k -\{ (i-2) + (i-4) \}$ choices for $\ell_{j_{i-1}}$ if one of $\sing(\mathcal{L}) \cap \ell_{j_{i-2}} \cap \ell_{j_1}$, \dots, $\sing(\mathcal{L}) \cap \ell_{j_{i-2}} \cap \ell_{j_{i-3}}$ is a double point. Also, some or all of $\sing(\mathcal{L}) \cap \ell_{j_3} \cap \ell_4$, \dots, $\sing(\mathcal{L}) \cap \ell_{j_{i-4}} \cap \ell_{j_{i-3}}$ might be triple points.So, we have to further remove at most $i-6$ possibilities and finally we have at least $k -\{ (i-2)+(i-4) +(i-6)\}$ possibilities for $\ell_{j_{i-1}}$. If none of  $\sing(\mathcal{L}) \cap \ell_{j_{i-2}} \cap \ell_{j_1}$, \dots, $\sing(\mathcal{L}) \cap \ell_{j_{i-2}} \cap \ell_{j_{i-3}}$ are double points then as in the previous cases we choose $\ell_{j_{i-1}}$ such that $\sing(\mathcal{L}) \cap \ell_{j_{i-2}} \cap \ell_{j_{i-1}}$ is a double point.

 For $\ell_{j_i}$ we have at least $k -\{ (i-1) + (i-3) + (i-5)\}$ choices if one of $\sing(\mathcal{L}) \cap \ell_{j_{i-1}} \cap \ell_{j_1}$, \dots, $\sing(\mathcal{L}) \cap \ell_{j_{i-1}} \cap \ell_{j_{i-2}}$ is a double point. Now, for an induced cycle of length $2i$ ending with $\ell_{j_i}$ and $\ell_{j_1}$ we also have to look at the following list of intersection points: 
 
 $\sing(\mathcal{L}) \cap \ell_{j_1} \cap \ell_{j_3}$, \dots, $\sing(\mathcal{L}) \cap \ell_{j_1} \cap \ell_{j_{i-2}}$

 Some or all of these points might be triple points. In that case, we have to further remove at most $i-4$ choices for $\ell_{j_i}$. So, we have at least $k -\{(i-1) +(i-3) +(i-5) +(i-4) = k -(4i-13)\}$ choices for $\ell_{j_i}$.

 If all of $\sing(\mathcal{L}) \cap \ell_{j_{i-1}} \cap \ell_{j_1}$, \dots, $\sing(\mathcal{L}) \cap \ell_{j_{i-1}} \cap \ell_{j_{i-2}}$ are triple points then we have at least $k - \{(i-1) +(i-2) +(i-5) +(i-4)\} = k -(4i - 12)$ choices for $\ell_{j_i}$. 
 
 Combining both the cases, we conclude that we have at least $k - (4i-12)$ choices for $\ell_{j_i}$. Now for $i= \lfloor {\frac{k+11}{4}} \rfloor$ there are at least $k - (4 \lfloor {\frac{k+11}{4}} \rfloor -12) \geq 1$ choices for $\BFC_{j_{\lfloor {\frac{k+11}{4}} \rfloor}}$. That proves the existence of an induced cycle of length $2i$ for every $ i \leq \lfloor {\frac{k+11}{4}} \rfloor$.
 
\vspace{2mm}
 \textit{Proof of $(iii)$:} All the lines in $\mathcal{L}$ has exactly one double point in it.

 Without loss of generality we can assume that the double point on $\ell_{j_1}$ lies on $\ell_{j_2}$ i.e. $\sing(\mathcal{L}) \cap \ell_{j_1} \cap \ell_{j_2}$ is a double point. Then we have $k-2$ choices for $\ell_{j_3}$. Since each line has only one double point on it, $\sing(\mathcal{L}) \cap \ell_{j_2} \cap \ell_{j_3}$ is not a double point. Choose, $j_4$ such that $\sing(\mathcal{L}) \cap \ell_{j_3} \cap \ell_{j_4}$ is a double point. 

 Since $\sing(\mathcal{L}) \cap \ell_{j_2} \cap \ell_{j_3}$ is a triple point, we have at least $k- (4+1)$ choices for $\ell_{j_5}$.

 Now, if $i$ is even i.e. $(i-1)$ odd, we can choose $\sing(\mathcal{L}) \cap \ell_{j_1} \cap \ell_{j_2}$, $\sing(\mathcal{L}) \cap \ell_{j_3} \cap \ell_{j_4}$, \dots, $\sing(\mathcal{L}) \cap \ell_{j_{i-3}} \cap \ell_{j_{i-2}}$ to be double points. Hence, $\sing(\mathcal{L}) \cap \ell_{j_2} \cap \ell_{j_3}$, $\sing(\mathcal{L}) \cap \ell_{j_4} \cap \ell_{j_5}$, \dots, $\sing(\mathcal{L}) \cap \ell_{j_{i-2}} \cap \ell_{j_{i-1}}$ are all triple points. 
 Then counting as in \textbf{Case I} we have at least $k -\{(i-1) +\frac{i-2}{2} +(i-3) +(i-4) \}= k - (\frac{7i-18}{2}) $ choices for $\ell_{j_i}$.

 If $i$ is odd, we can choose $\sing(\mathcal{L}) \cap \ell_{j_1} \cap \ell_{j_2}$, $\sing(\mathcal{L}) \cap \ell_{j_3} \cap \ell_{j_4}$, \dots, $\sing(\mathcal{L}) \cap \ell_{j_{i-2}} \cap \ell_{j_{i-1}}$ to be double points and $\sing(\mathcal{L}) \cap \ell_{j_2} \cap \ell_{j_3}$, $\sing(\mathcal{L}) \cap \ell_{j_4} \cap \ell_{j_5}$, \dots, $\sing(\mathcal{L}) \cap \ell_{j_{i-3}} \cap \ell_{j_{i-2}}$ to be all triple points. So, we have at least $k -\{(i-1) +\frac{i-3}{2} +(i-3) +(i-4) \}= k - (\frac{7i-19}{2})$ choices for $\ell_{j_i}$.

 combining both the cases, we observe that there are at least $ k - (\frac{7i-18}{2})$ choices for $\ell_{j_i}$. Hence for $i= \lfloor {\frac{2k+16}{7}} \rfloor$ there are at least $k - (7 \lfloor {\frac{2k+16}{7}} \rfloor -18) \geq 1$ choices for $\BFC_{j_{\lfloor {\frac{2k+16}{7}} \rfloor}}$. That proves the existence of an induced cycle of length $2i$ for every $ i \leq \lfloor {\frac{2k+16}{7}} \rfloor$.
 \end{proof}

We notice  in Theorem \ref{thm: max t_3} that  the associated Levi graphs of line arrangements of $k > 10$ lines contain induced cycles of length $10$. Actually we prove more than that. We show that the associated Levi graphs contain induced cycles of length $2i$, where $i$ increases with $k$. But this helps us very little to understand the length of longest induced cycles of a line arrangement in general. Determining length of longest induced cycles of a line arrangement depends heavily on its geometric structure and combinatorial data. We consider two examples with $t_r =0$ for all $r > 3$ and find out the length of longest induced cycles in each cases.

\vskip 2mm

\begin{example}\label{exam: nine lines with t_3}
Let $\mathcal{L}_{\{9\}}$ be a $(9_3)$ arrangement in $\mathbb{P}^2$ (arrangement of nine lines with $t_3=9$) where the lines are given by:
\begin{align*}
  L_1: y=bz, L_2: y=z, L_3 = y=0 \\
  L_4 : x=0, L_5: x=z, L_6: x=az \\
  L_7 : -x+z,L_8:  y =\frac{b}{a}x \, \, \, \, \hspace{1cm} \\
  L_9 : y= \frac{b(b-1)}{b-a}(x + \frac{b(1-a)}{b-a}z)
\end{align*}
where $a,b$ satisfying the relation $a - (b^2 -b +1) =0$. 

Triple points of $\sing(\mathcal{L}_{\{9\}})$ are given by the following table \ref{table:1}.
\begin{table}[h!]
\centering
\begin{tabular}{|ccc|ccc|ccc|}
\hline
$L_1$ &	$L_2$ &	$L_3$ & $L_4$ &	$L_5$ &	$L_6$ & $L_7$ &	$L_8$ & $L_9$\\
\hline
$e_1$ & $e_1$ & $e_1$ & $e_8$	&	$e_8$ & $e_8$ & $e_4$ & $e_3$	&	$e_2$ \\
$e_2$ & $e_4$ & $e_6$ & $e_4$	&	$e_2$ & $e_3$ & $e_7$ & $e_5$	&	$e_5$  \\
$e_3$ & $e_5$ & $e_7$ & $e_6$ &	$e_7$ & $e_9$ & $e_9$ & $e_6$	&	$e_9$ \\
\hline
\end{tabular}
\caption{ $9_3$ }
\label{table:1}
\end{table}

Remaining points of $\sing(\mathcal{L}_{\{ 9 \}})$ are double points and are denoted as:
\begin{align*}
    \sing(\mathcal{L}_{\{9\}}) \cap L_1 \cap L_4 = e_{10}, \sing(\mathcal{L}_{\{9\}}) \cap L_1 \cap L_7 = e_{11} \sing(\mathcal{L}_{\{9\}}) \cap L_2 \cap L_5 = e_{12}, \\ \sing(\mathcal{L}_{\{9\}}) \cap L_2 \cap L_6 = e_{13}, 
    \sing(\mathcal{L}_{\{9\}}) \cap L_3 \cap L_6 = e_{14},
    \sing(\mathcal{L}_{\{9\}}) \cap L_3 \cap L_9 = e_{15},\\
    \sing(\mathcal{L}_{\{9\}}) \cap L_4 \cap L_9 = e_{16} ,
    \sing(\mathcal{L}_{\{9\}}) \cap L_5 \cap L_8 = e_{17},
    \sing(\mathcal{L}_{\{9\}}) \cap L_7 \cap L_8 = e_{18}.
\end{align*}

By part $(i)$ of Theorem \ref{thm: max t_3} we see that the associated Levi graph of $\mathcal{L}_{\{9\}}$ has induced cycles of length $2i$ for $i \leq 4$. But we claim that the Levi graph $G_{\{9\}}$ associated to $\mathcal{L}_{\{9\}}$ contain induced cycles $\BFC_{2i}$ for $i \leq 7$ and length of the maximum induced cycle is $14$. First we show the existence of induced cycles of length $10,12$ and $14$.

 $\bullet$  $L_1e_{10}L_4e_4L_7e_{18}L_8e_5L_9e_2L_1$ corresponds to an induced cycle of length $10$ in $G_{\{9\}}$.
 
$\bullet$  $L_1e_{10}L_4e_4L_7e_{18}L_8e_5L_9e_{15}L_3e_1L_1$ corresponds to an induced cycle of length $12$ in $G_{\{9\}}$.

$\bullet$  $L_1e_1L_2e_{12}L_5e_7L_7e_9L_9e_{15}L_4e_6L_8e_3L_1$ corresponds to an induced cycle of length $14$ in $G_{\{9\}}$.

Now assume that $G_{\{9\}}$ contains an induced cycle of length $16$. Let $\ell_1p_1\ell_2\dots p_7\ell_8p_8\ell_1$ corresponds to an induced cycle of length $16$ in $G_{\{9\}}$. We first show that all $p_i$ for $i=1, \dots , 8$ cannot be double points.

Since, all the lines in $\mathcal{L}_{\{9\}}$ has two double points in it, w.l.o.g we can take $\ell_1 = L_1$. Next we choose $\ell_2 = L_4, p_1= e_{10}$. Then we must have $\ell_3= L_9$ and $p_2= e_{16}$. Similarly for $\ell_4,\ell_5, \ell_6, \ell_7$ and $\ell_8$ we then have $\ell_4= L_3, p_3=e_{15}, \ell_5= L_6, p_4 = e_{14}, \ell_6 =L_2, p_5=e_{13}, \ell_7 =L_5, p_6= e_{12}$ and $\ell_8 = L_8, p_7 = e_{17}$. But the intersection point of  $L_1$ and $L_8$ is $e_3$, a triple point. In fact, $e_3$ passes through $L_6$. So all the $p_i$ for $i=1, \dots, 8$ can not be double points. 

Now, w.l.o.g let $p_1=e_1$ be a triple point and let $\ell_1 = L_1$ and $\ell_2=L_2$. Now, $p_2$ cannot be a triple point, since any triple point on $L_2$ other than $e_1$ will block the entry of some $L_i$ for $i \neq 3$, which would make it impossible to form an induced cycle of length $16$. So, we have two choices for $\ell_3$, $L_5$ and $L_6$. If $\ell_3= L_6$, for the same reason as stated above $p_3$ cannot be triple point. If $p_3$ is a double point, only possibility for $\ell_4$ is $L_3$. This can not happen since $L_3$ passes through $e_1$. Hence $\ell_4 \neq L_3$ which further implies that $\ell_3 \neq L_6$. This leaves us with only one choice, $\ell_3 = L_5$. Now we show that $\ell_4= L_7$ and $p_3= e_7$, the intersection point of $L_3, L_5$ and $L_7$.

We observe that for the same reason a stated for $p_2$, $p_8$ cannot be a triple point. So, $\ell_8=L_4$ and $p_8=e_{10}$. There are nine double points in $\sing(\mathcal{L}_{\{9\}})$ and $\{p_i, \, i=1, \dots, 8\} \cap \{e_{11}, e_{14}, e_{15}\} = \emptyset$. So, apart from $e_1$ at least one more $p_i$ for some $i=1, \dots, 8$ must be a triple point and only possibilities are $e_6$ and $e_7$. Since $e_6$ is the intersection point of $L_3, L_4$ and $L_8$ and $\ell_8= L_4$, $e_6 \neq p_i$ for $i=1, \dots, 8$. So, $p_3 = e_7$ and $\ell_4=L_7$. Then, $\ell_5$ must be equal to $L_8$ and $p_4 = e_{18}$. Now, $p_5$ can not be a triple point since $e_7$ is not allowed. If $p_5$ is a double point, only choice for $p_5$ is $e_{17}$, which implies that $\ell_5 = L_5$, a contradiction. Hence, $G$ does not contain an induced cycle of length $16$ or more.

\end{example}

\begin{example}\label{exam: ten lines with t_2 and t_3}

If we add a tenth line $L_{10} : y = \frac{1}{a-1}(x-1)$ to the line arrangement in Example \ref{exam: nine lines with t_3}, we get a line arrangement $\mathcal{L}_{\{10\}}$ with $12$ triple points and $9$ double points. We denote the nine triple points $e_1, \dots, e_9$ as in table (\ref{table:1}) and the remaining triple points as follows:
\begin{align*}
    e_{10} = \sing(\mathcal{L}_{\{10\}}) \cap L_2 \cap L_5 \cap L_{10} \\
    e_{11} = \sing(\mathcal{L}_{\{10\}}) \cap L_3 \cap L_6 \cap L_{10} \\
    e_{12} = \sing(\mathcal{L}_{\{10\}}) \cap L_4 \cap L_9 \cap L_{10} .
\end{align*}
 All the remaining intersection points are double points. They are
 \begin{align*}
     e_{13} = \sing(\mathcal{L}_{\{10\}}) \cap L_1 \cap L_4 , e_{14} = \sing(\mathcal{L}_{\{10\}}) \cap L_1 \cap L_7, e_{15} = \sing(\mathcal{L}_{\{10\}}) \cap L_1 \cap L_{10} \\
     e_{16} = \sing(\mathcal{L}_{\{10\}}) \cap L_2 \cap L_6, e_{17} = \sing(\mathcal{L}_{\{10\}}) \cap L_3 \cap L_9, e_{18} = \sing(\mathcal{L}_{\{10\}}) \cap L_5 \cap L_8 \\
     e_{19} = \sing(\mathcal{L}_{\{10\}}) \cap L_7 \cap L_8, e_{20} = \sing(\mathcal{L}_{\{10\}}) \cap L_7 \cap L_{10}, e_{21} = \sing(\mathcal{L}_{\{10\}}) \cap L_8 \cap L_{10}.
 \end{align*}

 By part $(ii)$ of Theorem \ref{thm: max t_3} we conclude that the associated Levi graph $G_{\{10\}}$ of $\mathcal{L}_{\{10\}}$ has induced cycles of length $2i$ for $i \leq 5$. But using the combinatorial properties of $\mathcal{L}_{\{10\}}$ we claim that $G_{\{10\}}$ has induced cycles of length $2i$ for $i \leq 9$ and the maximum length of an induced cycle in $G_{\{10\}}$ is $18$. First we show the existence of induced cycles of lengths $2i$ for $6 \leq i \leq 9$ in $G_{\{10\}}$.

 $\bullet$ $L_1e_1L_3e_{17}L_9e_5L_8e_{21}L_{10}e_{12}L_4e_{13}L_1$ corresponds to an induced cycle of length $12$ in $G_{\{10\}}$.

 $\bullet$ $L_1e_{15}L_{10}e_{11}L_6e_{16}L_2e_5L_8e_{18}L_5e_7L_7e_{14}L_1$ corresponds to an induced cycle of length $14$ in $G_{\{10\}}$. 

 $\bullet$ $L_1e_1L_3e_{19}L_9e_5L_8e_{21}L_{10}e_{10}L_5e_8L_4e_4L_7e_{14}L_1$ corresponds to an induced cycle of length $16$ in $G_{\{10\}}$. 

 $\bullet$ $L_2e_{10}L_5e_{18}L_8e_{19}L_7e_{14}L_1e_{13}L_4e_{12}L_9e_{17}L_3e_{11}L_6e_{16}L_2$ corresponds to an induced cycle of length $18$ in $G_{\{10\}}$.

 As $\sing(\mathcal{L}_{\{10\}})$ has nine double intersection points, $\mathcal{L}_{\{10\}}$ can't have an induced cycle of length $20$. Hence, the maximum length of induced cycle in $G_{\{10\}}$ is $18$.

 \end{example}

.


 Now we reformulate Theorem \ref{thm: max t_3} in a more general setup.

   \begin{theorem}\label{thm: max t_q}
       Let $\mathcal{L} = \{ \ell_1, \dots \ell_k \}$ be a line arrangement of $k$ lines with $t_q \neq 0$ and $t_r = 0$ for all $r > q$. Then

    $(i)$ then the associated Levi graph has an induced $\BFC_{2i}$, for all $ i \leq \lfloor {\frac{k+ 9q -18}{3q-5}} \rfloor$.

    $(ii)$ If only non-zero entries in $(t_1, \dots, t_q)$ are $t_q$ and $t_p$ for some $p \in\{ 2, \dots, q-1\}$ and $q-1 \nmid k-1$, the associated Levi graph has an induced $\BFC_{2i}$ for all $ i \leq \lfloor {\frac{k+ 10q -18}{3q-5}} \rfloor$.

   \end{theorem}
   \begin{proof}
   Let us assume that $\sing{L}$ has only $q$-fold points i.e. $t_r = 0$ for all $r < q$. this is only possible if $k$ is a multiple of $q(q-1)$.
       The proof is similar to the proof of $(i)$ of Theorem \ref{thm: max t_3}. We provide some details. 
       Let $\ell_{j_1}, \dots , \ell_{j_i}$ corresponds to an induced cycle $\BFC_{2i}$ in $G$ and let us choose $\ell_{j_1}$ and $\ell_{j_2}$ as first two lines. Then for third line $\ell_{j_3}$ we have $k - \{2+(q-2)\}$ choices and for fourth line $\ell_{j_4}$ we have $k - \{3 + 2(q-2) +(q-2)\}$ choices. Similarly for an induced cycle of length $2i$ we have $k -\big\{(i-1)+\{(i-2)+(i-3) +(i-4)\}(q-2) \big\}$ choices for $\ell_{j_i}$. This proves the existence of an induced cycle of length $2i$ for every $ i \leq \lfloor {\frac{k+ 9q -18}{3q-5}} \rfloor$ .

       Proof of $(ii)$: $q-1 \nmid k-1$ in the hypothesis implies that every line $\ell_i$ in $\mathcal{L}$ has at least one $p$-point in it. 
       
      \textbf{Case I:}  At least one line in $\mathcal{L}$ has more than one $p$-point in it. 
      
      The proof in this case is similar to part $(ii)$ of Theorem \ref{thm: max t_3}.

       W.l.o.g assume that $\ell_{j_2}$
 has two $p$-points in it and these points are $\sing(\mathcal{L}) \cap \ell_{j_1} \cap \ell_{j_2}$ and $\sing(\mathcal{L}) \cap \ell_{j_2} \cap \ell_{j_3}$. So, we have at least $k - \{3+ 2(p-2) +(q-2)\}$ choices for $\ell_{j_4}$.  If any one of $\sing(\mathcal{L}) \cap \ell_{j_4} \cap \ell_{j_1}$, $\sing(\mathcal{L}) \cap \ell_{j_4} \cap \ell_{j_2}$ or $\sing(\mathcal{L}) \cap \ell_{j_4} \cap \ell_{j_3}$ is a $p$-point, we have $k - \{ 4+ 3(p-2) +2(q-2)\}$ choices for $\ell_{j_5}$. Otherwise, we choose $\sing(\mathcal{L}) \cap \ell_{j_4} \cap \ell_{j_2}$ a $p$-point. Proceeding similarly to $(ii)$ of Theorem \ref{thm: max t_3} we observe that if we wish to have an induced cycle of length $2i$ corresponding to $\ell_{j_1}. \dots, \ell_{j_i}$ then we have at least $k - \{(i-1) + 2(p-2) + (3i-11)(q-2)\} $ choices for $\ell_{j_i}$. Hence, the associated Levi graph $G$ of $\mathcal{L}$ has induced cycles of length $2i$ for all $ i \leq \lfloor {\frac{k -2p +11q -18}{3q-5}} \rfloor $.

\textbf{Case II:} All the lines in $\mathcal{L}$ has exactly one $p$-point in it. 

 We proceed as in the proof of $(iii)$ of Theorem \ref{thm: max t_3}. W.l.o.g we assume that $\sing(\mathcal{L}) \cap \ell_{j_1} \cap \ell_{j_2}$ is a $p$-point. Then we have $k-\{2 +(p-2)\}$ choices for $\ell_{j_3}$. Since, $\sing(\mathcal{L}) \cap \ell_{j_2} \cap \ell_{j_3}$ must be a $q$-point, we can choose $\ell_{j_4}$ such that $\sing(\mathcal{L}) \cap \ell_{j_3} \cap \ell_{j_4}$ is a $p$-point. So, we have at least $k -\{4 + 2(p-2) + 3(q-2)\}$ choices for $\ell_{j_5}$. 
 If $i$ is even we can choose  $\sing(\mathcal{L}) \cap \ell_{j_1} \cap \ell_{j_2}$, $\sing(\mathcal{L}) \cap \ell_{j_3} \cap \ell_{j_4}$, \dots, $\sing(\mathcal{L}) \cap \ell_{j_{i-3}} \cap \ell_{j_{i-2}}$ to be $p$-points and as a consequence $\sing(\mathcal{L}) \cap \ell_{j_2} \cap \ell_{j_3}$, $\sing(\mathcal{L}) \cap \ell_{j_4} \cap \ell_{j_5}$, \dots, $\sing(\mathcal{L}) \cap \ell_{j_{i-2}} \cap \ell_{j_{i-1}}$ all $q$-points. So, to have an induced cycle of length $2i$ we have $k -\{ (i-1) + \frac{(i-2)(p-2)}{2} + (\frac{5i-16}{2})(q-2)\}$ choices for $\ell_{j_i}$.

 If $i$ is odd then we have $k -\{(i-1) + \frac{(i-1)(p-2)}{2} + (\frac{5i-17}{2})(q-2)\}$ choices for $\ell_{j_i}$.

 Since $p < q$, to have an induced cycle of length $2i$ corresponding to $\ell_{j_1}, \dots, \ell_{j_i}$ we have at least $k -\{ (i-1) + \frac{(i-2)(p-2)}{2} + (\frac{5i-16}{2})(q-2)\}$ choices for $\ell_{j_i}$. So, the associated Levi graph of $\mathcal{L}$ has induced cycles of length $2i$ for all $i \leq \lfloor {\frac{2(k+p+8q -18)}{p+5q-10}} \rfloor$.
       
   \end{proof}

Next, we consider some line arrangements in $\mathbb{P}_{\mathbb{C}}^2$ with the largest multiplicity $q > 3$ and look into the existence of induced cycles in the associated Levi graphs.

\vskip 3mm

 \textbf{Hesse arrangement of lines}:

 Hesse configuration $\mathcal{L}_{\mathcal{H}}$ is a line arrangement of $12$ lines with $t_2 =12, t_4 = 9$ and $t_r =0$ otherwise.  

  \begin{figure}[H]
				\begin{tikzpicture}[scale=1.5]
\draw [line width=2pt] (-3.3,2)-- (0,2);
\draw [line width=2pt] (-1,2.3)-- (-1,-1);
\draw [line width=2pt] (-4,0)-- (-0.7,0);
\draw [line width=2pt] (-3,-0.3)-- (-3,3);
\draw [line width=2pt] (-3.3,1)-- (-0.7,1);
\draw [line width=2pt] (-3.25,-0.25)-- (-0.75,2.25);
\draw [line width=2pt] (-3.25,2.25)-- (-0.75,-0.25);
\draw [line width=2pt] (-2,2.3)-- (-2,-0.3);
\draw [line width=2pt][color=brown] (-1,1)-- (-2.25,-0.25);
\draw [line width=2pt][color=green] (-2,0)-- (-3.25,1.25);
\draw [line width=2pt][color=blue] (-3,1)-- (-1.75,2.25);
\draw [line width=2pt][color=red] (-2,2)-- (-0.75,.75);
\draw [line width=2pt][color=brown] (-3,2)-- (-3.25,1.75);
\draw [line width=2pt][color=green] (-1,2)-- (-1.25,2.25);
\draw [line width=2pt][color=blue] (-1,0)-- (-0.75,0.25);
\draw [line width=2pt][color=red] (-3,0)-- (-2.75,-0.25);
\draw [color=brown][shift={(-1.6388888888888888,2.2222222222222223)},line width=2pt]  plot[domain=-1.0891326492400353:3.303430084828216,variable=\t]({1*1.3791323985896011*cos(\t r)+0*1.3791323985896011*sin(\t r)},{0*1.3791323985896011*cos(\t r)+1*1.3791323985896011*sin(\t r)});
\draw [color=green][shift={(-0.7694877049180329,0.6347438524590164)},line width=2pt]  plot[domain=-2.6653558114207323:1.7380605934191202,variable=\t]({1*1.3845794547740493*cos(\t r)+0*1.3845794547740493*sin(\t r)},{0*1.3845794547740493*cos(\t r)+1*1.3845794547740493*sin(\t r)});
\draw [color=red][shift={(-3.2405482456140353,1.3702741228070177)},line width=2pt]  plot[domain=0.46972422600434666:4.886165863173627,variable=\t]({1*1.3912277419971655*cos(\t r)+0*1.3912277419971655*sin(\t r)},{0*1.3912277419971655*cos(\t r)+1*1.3912277419971655*sin(\t r)});
\draw [color=blue][shift={(-2.3675510204081633,-0.23510204081632663)},line width=2pt]  plot[domain=-4.239138298833666:0.17025042724226078,variable=\t]({1*1.3876126127328947*cos(\t r)+0*1.3876126127328947*sin(\t r)},{0*1.3876126127328947*cos(\t r)+1*1.3876126127328947*sin(\t r)});
\begin{scriptsize}
\fill  (-3,2) circle (3pt);
\fill  (-2,2) circle (3pt);
\fill  (-1,2) circle (3pt);
\fill  (-3,1) circle (3pt);
\fill  (-2,1) circle (3pt);
\fill  (-1,1) circle (3pt);
\fill  (-3,0) circle (3pt);
\fill  (-2,0) circle (3pt);
\fill  (-1,0) circle (3pt);
\draw (-3.4,2) node {$\ell_1$};
\draw (-4.1,0) node {$\ell_3$};
\draw (-3.4,1) node {$\ell_2$};
\draw (-3,-0.4) node {$\ell_4$};
\draw (-2,-0.4) node {$\ell_5$};
\draw (-1,-1.1) node {$\ell_6$};
\draw (-3.3,-0.3) node {$\ell_7$};
\draw (-0.65,-0.3) node {$\ell_8$};
\draw (-1.5,3.4) node {$\ell_9$};
\draw (0.4,0.45) node {$\ell_{10}$};
\draw (-4.4,1.5) node {$\ell_{11}$};
\draw (-2.5,-1.4) node {$\ell_{12}$};
\end{scriptsize}
              \end{tikzpicture}
				\caption{Hesse arrangement}\label{fig: Hesse}
			\end{figure}
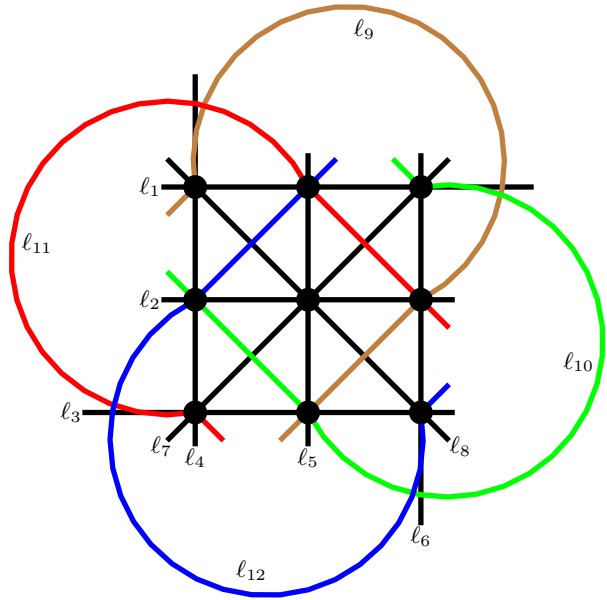

 In figure \ref{fig: Hesse} above, let us denote the horizontal lines as $\ell_1, \ell_2, \ell_3$ from top to down, the vertical lines as $\ell_4, \ell_5, \ell_6$ from left to right, diagonals as $\ell_7, \ell_8$ and broken diagonals as $\ell_9,\ell_{10}, \ell_{11}, \ell_{12}$ clockwise.

 The $4$-points of $\mathcal{L}_{\mathcal{H}}$ are:

 \begin{align*}
     \sing(\mathcal{L}_{\mathcal{H}}) \cap \ell_1 \cap \ell_4 \cap \ell_7\cap \ell_{10} = \{p_1\}, \sing(\mathcal{L}_{\mathcal{H}}) \cap \ell_1 \cap \ell_5 \cap \ell_9\cap \ell_{12} = \{p_2\}, \\  \sing(\mathcal{L}_{\mathcal{H}}) \cap \ell_1 \cap \ell_6 \cap \ell_8\cap \ell_{11} = \{p_3\},
     \sing(\mathcal{L}_{\mathcal{H}}) \cap \ell_2 \cap \ell_4 \cap \ell_{11}\cap \ell_{12} = \{p_4\}, \\ \sing(\mathcal{L}_{\mathcal{H}}) \cap \ell_2 \cap \ell_5 \cap \ell_7\cap \ell_8 = \{p_5\}, \sing(\mathcal{L}_{\mathcal{H}}) \cap \ell_2 \cap \ell_6 \cap \ell_9\cap \ell_{10} = \{p_6\}, \\
     \sing(\mathcal{L}_{\mathcal{H}}) \cap \ell_3 \cap \ell_4 \cap \ell_8\cap \ell_9 = \{p_7\}, \sing(\mathcal{L}_{\mathcal{H}}) \cap \ell_3 \cap \ell_5 \cap \ell_{10}\cap \ell_{11} = \{p_8\}, \\
     \sing(\mathcal{L}_{\mathcal{H}}) \cap \ell_3 \cap \ell_6 \cap \ell_7\cap \ell_{12} = \{p_9\}.
     \end{align*}
     Remaining elements of $\sing(\mathcal{L}_{\mathcal{H}})$ are double points. Whenever $\sing(\mathcal{L}_{\mathcal{H}}) \cap \ell_i \cap \ell_j$ is a double point, we denote it by $p_{ij}$.

      Note that 
      \begin{align*}
          \{p_1, \dots, p_9\} \subseteq \{ \sing(\mathcal{L}_{\mathcal{H}}) \cap \ell_{r_1}, \sing(\mathcal{L}_{\mathcal{H}}) \cap \ell_{r_2}, \sing(\mathcal{L}_{\mathcal{H}}) \cap \ell_{r_3} \}
      \end{align*}

       where $(r_1,r_2,r_3)$ is one of $(1,2,3), (4,5,6), (7,9,11)$ or $(8,10,12)$. Let us assume that $(j_1,j_2,j_3) = (1,2,3)$.

     We give explicit descriptions of induced cycles $\BFC_{10}, \BFC_{12}$ and show that there does not exists induced cycles $\BFC_{2i}$ for $i  \geq 7$. 

     \begin{theorem}\label{thm: hesse arrangement}
         The Levi graphs associated to Hesse arrangement $\mathcal{L}_{\mathcal{H}}$ contains induced cycles of length $2i$ for $i \leq 6$ and the length of longest induced cycles is $12$.
     \end{theorem}

     \begin{proof}

      We observe that $\ell_3p_{13}\ell_1p_{12}\ell_2p_4\ell_4p_{45}\ell_5p_8\ell_3$ corresponds to an induced $\BFC_{10}$ in the associated Levi graph $G_{\mathcal{H}}$ and $\ell_3p_{13}\ell_1p_{12}\ell_2p_4\ell_4p_{45}\ell_5p_{56}\ell_6p_9\ell_3$ corresponds to an induced $\BFC_{12}$ in the associated Levi graph $G_{\mathcal{H}}$.

     Suppose there exists an induced $\BFC_{2i}$ for $i \geq 7$ corresponding to $\ell_{j_1}, \dots, \ell_{j_i}$. Let
  $S = \{ \sing(\mathcal{L}_{\mathcal{H}}) \cap \ell_{j_t}\cap \ell_{j_{t+1}}, \, \, t=1, \dots, i-1, \, \sing(\mathcal{L}_{\mathcal{H}}) \cap \ell_{j_{i-1}} \cap \ell_{j_i} \}$. we say two points $p$ and $p'$ of $S$ are adjacent if $p = \sing(\mathcal{L}_{\mathcal{H}}) \cap \ell_{j_{t-1}} \cap \ell_{j_t}$ and $p' = \sing(\mathcal{L}_{\mathcal{H}}) \cap \ell_{j_t} \cap \ell_{j_{t+1}} $ for some $t\in \{1, \dots, i\}$. Three points $p,p',p''$ of $S$ are said to be occur consecutively if $p,p'$ are adjacent and $p',p''$ are adjacent. Before going to the proof, we make three claims.
   \vskip 2mm

   $\bullet$ \textit{Claim I:} All points of $S$ can not be $4$-points.

   $\bullet$ \textit{Claim II:} There can be at most three double points in $S$ and no two of them are consecutive.

   $\bullet$ \textit{Claim III:} if $S$ contain at least one double point then four or more $4$-points of $S$ can't be consecutive and three $4$- points in $S$ can occur consecutively if one of the $4$-point is adjacent to a double point.

   \vskip 2mm

   \textit{Proof of Claim I:} Suppose all the points in $S$ are $4$-points. Since all the lines of $\mathcal{L}_{\mathcal{H}}$ have same number of double and $4$-points on them, wlog we assume that $\ell_{j_1} = \ell_1$. Now we know that all the $4$-points belong to $\sing(\mathcal{L}_{\mathcal{H}}) \cap \ell_{r_1}$ for $r_1 \in \{1,2,3\}$. If $j_t = 2$ for some $t \in \{ 2, \dots, i\}$, there will be at least three $4$-points in $S$ that lie on $\ell_3$ and at least two of them are adjacent. Let, $\sing(\mathcal{L}_{\mathcal{H}}) \cap \ell_{j_{s-1}} \cap \ell_{j_s}$ and $\sing(\mathcal{L}_{\mathcal{H}}) \cap \ell_{j_s} \cap \ell_{j_{s+1}}$ belong to $\sing(\mathcal{L}_{\mathcal{H}}) \cap \ell_3$ for some $s \in \{2, \dots, i\}$. Then $j_s =3$. There still exists a third $4$-point that belongs to $\sing(\mathcal{L}_{\mathcal{H}}) \cap \ell_3$, a contradiction. We can make similar conclusion if $j_t = 3$ for some $t \in \{2, \dots, i\}$. So, $\{j_2, \dots, j_i \} \cap \{2,3\} = \emptyset$. In the same way we can show that $|\{j_2, \dots, j_i \} \cap \{r_1,r_2,r_3\}| \leq 1$ for $(r_1,r_2, r_3\} = (4,5,6), (7,9,11), (8,10,12)$. This contradicts the fact that all points of $S$ are $4$-points. 

   \vskip 2mm

   \textit{Proof of Claim II:} Suppose there are four double points in $S$. Then at least two of these double points are consecutive. Without loss of generality, we assume that $\sing(\mathcal{L}_{\mathcal{H}}) \cap \ell_{j_1}\cap \ell_{j_2}$ and $\sing(\mathcal{L}_{\mathcal{H}}) \cap \ell_{j_2}\cap \ell_{j_3}$ are double points. Then $(j_1, j_2, j_3) \in \{ (1,2,3), (4,5,6), ( 7,9,11), (8,10,12)\}$. Let us assume that $(j_1,j_2,j_3) = (1,2,3)$. Then, none of $\sing(\mathcal{L}_{\mathcal{H}}) \cap \ell_{j_t}\cap \ell_{j_{t+1}}$ for $t = 4, \dots, i-1$ can be $4$-points, as all $4$-points pass through $\sing(\mathcal{L}_{\mathcal{H}}) \cap \ell_{r_1}$, for $r_1 \in \{1,2,3\}$.
     
      For the same reason, all of $\sing(\mathcal{L}_{\mathcal{H}}) \cap \ell_{j_4}\cap \ell_{j_5}, \sing(\mathcal{L}_{\mathcal{H}}) \cap \ell_{j_5}\cap \ell_{j_6}, \sing(\mathcal{L}_{\mathcal{H}}) \cap \ell_{j_6}\cap \ell_{j_7}$ must be double points. But then possibilities of  $(j_4, j_4, j_5)$ are  $(4,5,6)$ or $(7,9,11)$ or $(8, 10,12)$.
      
      Now $\sing(\mathcal{L}_{\mathcal{H}}) \cap \ell_{j_6}\cap \ell_{j_7}$ can't be a double point, since $\ell_{j_6}$ has only two double points in it and these are $\sing(\mathcal{L}_{\mathcal{H}}) \cap \ell_{j_6}\cap \ell_{j_5}$ and $\sing(\mathcal{L}_{\mathcal{H}}) \cap \ell_{j_6}\cap \ell_{j_4}$.
      
      Similar argument works when $(j_1,j_2,j_3) = (4,5,6)$ or  $(7,9,11)$ or $(8,10,12)$.

      So, $S$ has at most three double points. If any two of these double points are consecutive, similar argument as above draw a contradiction. Hence, no two of these double points are consecutive.

      \vskip 2mm
      
      \textit{Proof of Claim III:}
       Let  $\sing(\mathcal{L}_{\mathcal{H}}) \cap \ell_{j_1} \cap \ell_{j_2}$ be a double point. Then $j_1, j_2$ belongs to one of the following sets: $\{1,2,3\}, \{4,5,6\}, \{7,9,11\}, \{8,10,12\}$. Let us consider the case where $j_1, j_2 \in \{1,2,3\}$. Other cases can be proved similarly. Now, if $j_1,j_2 \in \{1,2,3\}, (j_1,j_2)$ can be one 
 of $(1,2), (1,3)$ or $(2,3)$. Let $(j_1, j_2) = (1,2)$. For other choices of $(j_1, j_2) proo$ As all $4$- points of $\sing(\mathcal{L}_{\mathcal{H}})$ pass through one of $\ell_1, \ell_2$ or $\ell_3$, any $4$-point in $S$ not adjacent to $\sing(\mathcal{L}_{\mathcal{H}}) \cap \ell_1 \cap \ell_2$ pass through $\ell_3$. Now assume that three $4$-points not adjacent to $\sing(\mathcal{L}_{\mathcal{H}}) \cap \ell_1 \cap \ell_2$ occur consecutively. Then all three of them pass through $\ell_3$, a contradiction. Similarly if four or more $4$- points occur consecutively, we have three consecutive $4$-points which are not adjacent to $\sing(\mathcal{L}_{\mathcal{H}}) \cap \ell_1 \cap \ell_2$. So, Four or more $4$-points in $S$ can not occur consecutively.

 Now,  let $S$ has at most two double points. . If $S$ has no double point, then we get a contradiction via \textit{Claim I}. If $S$ has one double point, we get a contradiction via \textit{Claim II}. If $S$ has two double points, at least three $4$-points occur consecutively. If these all of them are not adjacent to a double point, we are done. So, assume that three $4$-points of $S$ are consecutive and one of them is adjacent to a double point. 

 As in the proof of \textit{Claim II}, we assume that $\ell_{j_1} = \ell_1. \ell_{j_2} = \ell_2$ and $\sing(\mathcal{L}_{\mathcal{H}}) \cap \ell_1 \cap \ell_2$ is a double point. Now without loss of generality assume that three consecutive $4$-points are $\sing(\mathcal{L}_{\mathcal{H}}) \cap \ell_1 \cap \ell_{j_i}, \sing(\mathcal{L}_{\mathcal{H}}) \cap \ell_{j_i} \cap \ell_{j_{i-1}}$ and $\sing(\mathcal{L}_{\mathcal{H}}) \cap \ell_{j_{i-1}} \cap \ell_{j_{i-2}}$. Now, suppose $\ell_{j_{i-1}} \neq \ell_3$. Since both $\sing(\mathcal{L}_{\mathcal{H}}) \cap \ell_{j_i} \cap \ell_{j_{i-1}}$ and $\sing(\mathcal{L}_{\mathcal{H}}) \cap \ell_{j_{i-1}} \cap \ell_{j_{i-2}}$ pass through $\ell_3$, these two points belong to $\sing(\mathcal{L}_{\mathcal{H}}) \cap \ell_{j_{i-1}} \cap \ell_3$, which is not possible. So, $\ell_{j_{i-1}} = \ell_3$. Now since two double points of $S$ are not adjacent, $\sing(\mathcal{L}_{\mathcal{H}}) \cap \ell_{j_3} \cap \ell_{j_4}$ is a $4$-point. But $\sing(\mathcal{L}_{\mathcal{H}}) \cap \ell_{j_3} \cap \ell_{j_4}$ is on $\ell_1, \ell_2$ or $\ell_3$, a contradiction.

 Now for $i=7$ we assume that $S$ has exactly three double points and none of them are consecutive. Then, two $4$-points are consecutive. Let us also assume that $(j_1,j_2,j_3) = (1,2,3)$. So, both of the two consecutive $4$-points lie on one of $\ell_1, \ell_2$ or $\ell_3$. Without loss of generality, we assume that $\ell_{j_1} = \ell_1$ and $\sing(\mathcal{L}_{\mathcal{H}}) \cap \ell_1 \cap \ell_{j_2}$, $\sing(\mathcal{L}_{\mathcal{H}}) \cap \ell_1 \cap \ell_{j_7}$ are two adjacent $4$-points. There are few possibilities for $j_2$ and $j_7$. We consider these one by one.
      
  $\bullet$ Let us assume that $\sing(\mathcal{L}_{\mathcal{H}}) \cap \ell_1 \cap \ell_{j_2} =p_1$ and $\sing(\mathcal{L}_{\mathcal{H}}) \cap \ell_1 \cap \ell_{j_7} = p_2$. Now we have to make choices for $\ell_{j_2}$ and $\ell_{j_7}$. Since both $\sing(\mathcal{L}_{\mathcal{H}}) \cap \ell_{j_2} \cap \ell_{j_3}$ and $\sing(\mathcal{L}_{\mathcal{H}}) \cap \ell_{j_6}\cap \ell_{j_7}$ has to be double points, if we choose $\ell_{j_2} = \ell_4$, $\ell_{j_7}$ can't be equal to $\ell_5$. Hence, $\ell_{j_7}$ is either $\ell_9$ or $\ell_{12}$. First, we assume, $\ell_{j_2} = \ell_4$ and $\ell_{j_7} = \ell_9$. We know that two double points on $\ell_4$ are $\sing(\mathcal{L}_{\mathcal{H}}) \cap \ell_4 \cap \ell_5$ and $\sing(\mathcal{L}_{\mathcal{H}}) \cap \ell_4 \cap \ell_6$. Since, $p_2 \in \sing(\mathcal{L}_{\mathcal{H}}) \cap \ell_5$, we must have $\ell_{j_3} =  \ell_6$. Similarly, $\ell_9$ has two double points on it, $\sing(\mathcal{L}_{\mathcal{H}}) \cap \ell_9 \cap \ell_7$ and $\sing(\mathcal{L}_{\mathcal{H}}) \cap \ell_9 \cap \ell_{11}$.   Since $p_1 \in \sing(\mathcal{L}_{\mathcal{H}}) \cap \ell_7$, only possibility of $\ell_{j_6}$ is $\ell_{11}$. Since all the three double points are non-adjacent, $\sing(\mathcal{L}_{\mathcal{H}}) \cap \ell_6 \cap \ell_{j_4}$ and $\sing(\mathcal{L}_{\mathcal{H}}) \cap \ell_9 \cap \ell_{j_5}$ are $4$-points. As $p_6 \in \sing(\mathcal{L}_{\mathcal{H}}) \cap \ell_9$, $\ell_{j_4} \neq \ell_2, \ell_{10}$. So, we have to consider $p_9= \sing(\mathcal{L}_{\mathcal{H}}) \cap \ell_3 \cap \ell_6 \cap \ell_7 \cap \ell_{12}$ and the only choice for $\ell_{j_4}$ is $\ell_3$ in that case. For the choice of $\ell_{j_6}$, we know that $\sing(\mathcal{L}_{\mathcal{H}}) \cap \ell_3 \cap \ell_{j_5}$ is a double point. So, the only choice for $\ell_{j_5}$ is $\ell_2$, as we have already chosen $\ell_{j_1}$ to be $\ell_1$. Now, $\sing(\mathcal{L}_{\mathcal{H}}) \cap \ell_2 \cap \ell_{11}$ must be a $4$-point which is $p_4$. But $p_4$ is on $\ell_4$, which is $\ell_{j_2}$, a contradiction.

  Now assume that $\ell_{j_2} = \ell_4$ and $\ell_{j_7} = \ell_{12}$. As in the previous case, $\ell_{j_3} = \ell_6$. Now $\sing(\mathcal{L}_{\mathcal{H}}) \cap \ell_6 \cap \ell_{j_4}$ is a $4$-point and the possible choices are $p_6$ and $p_9$. But $p_9$ is on $\ell_{12}$. So, $\sing(\mathcal{L}_{\mathcal{H}}) \cap \ell_6 \cap \ell_{j_4}$ is $p_9$ and $\ell_{j_4} = \ell_2$. Also, $\ell_{j_6} \neq \ell_{10}$, as $p_1 \in \sing(\mathcal{L}_{\mathcal{H}}) \cap \ell_{10}$. So, $\ell_{j_6}$ has to be $\ell_8$. Next we can easily see that the only choice for $\ell_{j_5}$ is $\ell_3$. $j_5 = 3, j_6 = 8$ and $\sing(\mathcal{L}_{\mathcal{H}}) \cap \ell_3 \cap \ell_8$ is a $4$-point which is $p_7$. But, $p_7$ is on $\ell_4$, a contradiction.

  Other choices for $(j_2, j_7)$ are $(5,7), (7, 12), (5, 10)$ and $(9, 10)$. In those cases the proofs are similar and the conclusion is the same, hence omitted.

  Finally, we have have two more choices $(p_1, p_3)$ and $(p_2, p_3)$ for $\sing(\mathcal{L}_{\mathcal{H}}) \cap \ell_1 \cap \ell_{j_2},\sing(\mathcal{L}_{\mathcal{H}}) \cap \ell_1 \cap \ell_{j_7})$. In each case we reach a contradiction as observed above. The proofs are similar, hence the details are not illustrated. 

  For $i \geq 8$ with $S$ having three double points, we reach a contradiction in all cases via \textit{Claim I, Claim II} or \textit{Claim III} barring one case. This case is for $i=8$ and when $S$ contains three double points, two sets of adjacent $4$-points but no three $4$-points are consecutive. Mimicking the proof of $i=7$ with three double points in $S$, we can reach a contradiction in this case too.

  \end{proof}

 \textbf{Ceva's arrangements of lines:}
Let us consider the Ceva's arrangement $\mathcal{L}_{\mathcal{C}}$ of $3n$ lines given by:
$$(x^n-y^n)(y^n-z^n)(x^n-z^n).$$

 For $n \geq 4$, we have $t_3 = n^2, t_n =3$ and $t_r= 0$ elsewhere. In the theorem below, we study the induced cycles in the associated Levi graph $G_{\mathcal{C}}$.
 
\begin{theorem}
    For Ceva's lines arrangement $\mathcal{L}_{\mathcal{C}}$, the associated Levi graph $G_{\mathcal{C}}$ contain induced cycles $\BFC_{2i}$ for $4\leq i\leq 2n+1$.
\end{theorem}
\begin{proof}

Let us fix the following notations for the lines:
\begin{align*}
    L_{xy}^i: x-\ep^iy \text{ for } 0\leq i\leq n-1; \\
    L_{xz}^i: x-\ep^iz \text{ for } 0\leq i\leq n-1; \\
    L_{yz}^i: y-\ep^iz \text{ for } 0\leq i\leq n-1.
\end{align*}

    Let $\ell_1 p_1\ell_2 p_2\cdots \ell_{i-1}p_{i-1}\ell_i p_{i}\ell_1$ correspond to an induced cycle of length $2i$ for $4\leq i\leq 2n+1$. We now construct such lines to get induced cycle $\BFC_{2i}$. We divide into two cases:
    
    \vskip 2mm
    \noindent
    \textbf{Case-I:} Suppose $n$ is odd. We first show that $G$ contains an induced cycle of length $2m$, where $m$ is odd and $5\leq m\leq 2n+1$. Let us consider the lines $\mathcal{L}_{4i+3}:$ $\ell_1=L^1_{xy}$, $\ell_2=L^0_{xy}$, $\ell_3=L^0_{yz}$, $\ell_{4i}=L^{2i}_{xz}$, $\ell_{4i+1}=L^{2i+1}_{yz}$, $\ell_{4i+2}=L^{2i-1}_{xz}$, $\ell_{4i+3}=L^{2i}_{yz}$ for $1\leq i\leq (n-3)/2.$
    We can take either the lines $\mathcal{L}_{4i+1}$ with the end line $\ell_{4i+1}=L^{2i+1}_{yz}$ or the lines $\mathcal{L}_{4i+3}$ with the end line $\ell_{4i+3}=L^{2i}_{yz}$ for $1\leq i\leq (n-3)/2$ together with the corresponding intersection points to get a cycle $\BFC_{2m}$ of length $2m$, where $m$ is odd and $5\leq m\leq 2n-3$. 
    
    
    We make an induced cycle of length $2(2n-1)$ by considering the lines $\mathcal{L}_{2n-5}$ with the extra $4$ lines $\ell_{2n-4}=L^0_{xz}$, $\ell_{2n-3}=L^1_{yz}$, $\ell_{2n-2}=L^{n-4}_{xz}$ and $\ell_{2n-1}=L^{n-3}_{yz}$. We now make an induced cycle of maximum length $2(2n+1)$ by considering the lines $\mathcal{L}_{2n-3}$ with the extra lines $\ell_{2n-2}=L^{n-1}_{xz}$, $\ell_{2n-1}=L^1_{yz}$, $\ell_{2n}=L^{n-2}_{xz}$ and $\ell_{2n+1}=L^{n-1}_{yz}$.

    So, it remains to show that $G$ contains an induced cycle of length $2m$, where $m$ is even and $4\leq m \leq 2n$. Let $\BFC_{2m}: \ell_1 p_1\ell_2 p_2\cdots \ell_{m-1}p_{m-1}\ell_m p_{m}\ell_1$ correspond to an induced cycle of length $2m$, where $m$ is odd and $5\leq m\leq 2n+1$. We modify this cycle by considering $\ell_1' = \ell_1, \ell_2'= L^2_{xy}$ and $\ell_j' = \ell_{j+1}$ for $j \geq 3$. Then $\ell_1' p_1'\ell_2' p_2'\cdots \ell_{m-1}'p_{m-1}'\ell_m' p_{m}'\ell_1'$ corresponds to an induced cycle of length $2m$ for $m$ even and $4\leq m \leq 2n$. 
    We now show that the obtained cycle $\BFC_{2m}$ is an induced cycle.
    
    \noindent
    \textbf{Case-II:} Assume that $n$ is even. We first construct an induced cycle of length $2m$, where $m$ is odd and $5\leq m\leq (n-6)/2$. Consider the lines $\mathcal{L}_{4i+5}$: $\ell_1=L^1_{xy}$, $\ell_2=L^0_{xy}$, $\ell_3=L^0_{yz}$, $\ell_4=L^2_{xz}$, $\ell_5=L^3_{yz}$, $\ell_{4i+2}=L^{2i+3}_{xz}$, $\ell_{4i+3}=L^{2i+4}_{yz}$, $\ell_{4i+4}=L^{2i+2}_{xz}$, $\ell_{4i+5}=L^{2i+3}_{yz}$  for $1\leq i\leq (n-6)/2$. Then $\ell_1 p_1\ell_2 p_2\ell_3 p_3\ell_4 p_4 \ell_5 p_5\ell_1$ corresponds to an induced cycle of length $2\cdot 5$. For $m\geq 7$, we consider the lines $\mathcal{L}_{4i+3}$ with the end line $\ell_{4i+3}=L^{2i+4}_{yz}$ or $\mathcal{L}_{4i+5}$ with the end line $\ell_{4i+5}=L^{2i+3}_{yz}$ together with the corresponding intersection points to get an induced cycle $\BFC_{2m}$ of length $2m$ for $m$ odd and $7\leq m\leq 2n-7$. We make an induced cycle $\BFC_{2(2n-5)}$ by considering the lines $\mathcal{L}_{2n-9}$ with extra $4$ lines $\ell_{2n-8}=L^1_{xz}$, $\ell_{2n-7}=L^2_{yz}$, $\ell_{2n-6}=L^{2n-2}_{xz}$, $\ell_{2n-5}=L^{n-1}_{yz}$. Now we make an induced cycle of length $2(2n-3)$, $2(2n-1)$, $2(2n+1)$ from the induced cycle $\BFC_{2(2n-7)}$. Let $\ell_{2n-6}=L^{n-1}_{xz}$, $\ell_{2n-5}=L^1_{yz}$, $\ell_{2n-4}=L^3_{xz}$, $\ell_{2n-3}=L^4_{yz}$, $\ell_{2n-2}=L^1_{xz}$, $\ell_{2n-1}=L^2_{yz}$, $\ell_{2n}=L^{n-2}_{xz}$, $\ell_{2n+1}=L^{n-1}_{yz}$. Therefore,
    \begin{itemize}
        \item the lines of $\mathcal{L}_{2n-7}$ with extra lines 
 $\ell_{2n-6}$, $\ell_{2n-5}$, $\ell_{2n-4}$, $\ell_{2n-3}$, $\ell_{2n-2}$, $\ell_{2n-1}$, $\ell_{2n}$, $\ell_{2n+1}$ correspond to an induced cycle of length $2(2n+1)$.
 \item the lines of $\mathcal{L}_{2n-7}$ with extra lines 
 $\ell_{2n-6}$, $\ell_{2n-5}$, $\ell_{2n-4}$, $\ell_{2n-3}$, $\ell_{2n}$, $\ell_{2n+1}$ correspond to an induced cycle of length $2(2n-1)$.
 \item the lines of $\mathcal{L}_{2n-7}$ with extra lines 
 $\ell_{2n-6}$, $\ell_{2n-5}$, $\ell_{2n}$, $\ell_{2n+1}$ correspond to an induced cycle of length $2(2n-3)$.
    \end{itemize}

    Similar to the Case-I, if $\BFC_{2m}: \ell_1 p_1\ell_2 p_2\cdots \ell_{m-1}p_{m-1}\ell_m p_{m}\ell_1$ corresponds to an induced cycle of length $2m$ for $m$ being odd, then by modifying the cycle with the consideration of $\ell_1' = \ell_1, \ell_2'= L^2_{xy}$ and $\ell_j' = \ell_{j+1}$ for $j \geq 3$, we get an induced cycle of length $2m$, where $m$ is even and $4\leq m \leq 2n$.

    

\end{proof}

 \vspace{4mm}
 \textbf{Supersolvable line arrangements:}
A line arrangement $\mathcal{L}$ is supersolvable if $\sing(\mathcal{L})$ contains a modular point. We use the nations $\mu_{\mathcal{L}}$ and $m_{\mathcal{L}}$ for the number of modular points in $\sing(\mathcal{L})$ and for maximum multiplicity of modular points respectively. A supersolvable line arrangement is called non-homogenous if all the modular points don't have same multiplicity and is called $m$- homogeneous if all the modular points have the common multiplicity $m = m_{\mathcal{L}}$. 

We start with a simple yet important observation.

\vskip 2mm 

\begin{theorem}
 Let $\mathcal{L}$ be a supersolvable line arrangement of $k$ lines in $\mathbb{P}_\mathbb{C}^2$. Then the associated Levi graph $G$ does not have an induced cycle of maximum length i.e. $G$ does not contain an induced cycle of length $2k$.
\end{theorem}
\begin{proof}
First we assume that there exists an induced cycle of length $2k$. Let $\ell_1p_{12}\ell_2 \dots \ell_kp_{k1}\ell_1$ corresponds to an induced cycle of length $2k$ in $G$ where $p_{ij} \in \sing(\mathcal{L}) \cap L_i \cap L_j$.

Now let $p \in \sing(\mathcal{L})$ be a modular point. choose any element from the set $\{p_{i(i+1)}, \, \, 1 \leq i \leq k-1, p_{k1} \}$, say $p_{12}$. Since $p$ is a modular point, the line joining $p$ and $p_{12}$ is $\ell_v$ for some $v \in \{1, \dots, k\}$. If $v=1$ or $v=2$, $p=p_{12}$. Now, choose any point $p_{j(j+1)}$ for some $j \in \{2, \dots, k\}$. Since, $p_{12}$ is a modular point, the line joining $p_{12}$ and $p_{j(j+1)}$ is $\ell_{v_1}$ for some $v_1 \in \{1, \dots, k\}$. But, $p_{12} \in \sing(\mathcal{L}) \cap \ell_{v_1}$ implies that $v_1 = 1$ or $v_1=2$. Hence, $p_{j(j+1)}$ lies on either $\ell_1$ or $\ell_2$ which contradicts the fact that $\ell_1p_{12}\ell_2 \dots \ell_kp_{k1}\ell_1$ corresponds to an induced cycle of length $2k$.

If $v\neq 1,2$, $p_{12} \in \sing(\mathcal{L}) \cap \ell_v$, a contradiction to the assumption. Hence, $G$ does not have an induced cycle of length $2k$.
\end{proof}

Since the associated Levi graph of a supersolvable line arrangement does not have an induced cycle of maximum length, it is natural to ask the following question.

\textbf{Question:} what is the maximum length of an induced cycle in Levi graph associated to a supersolvable line arrangement?

 Based on the classification of supersolvable line arrangements in \cite{HH21} and \cite{AD20}, We study induced cycles of Levi graphs associated to supersolvable line arrangements. 

In \cite{HH21}, the authors gave the following description of supersolvable line arrangements (over any field) of $k$ lines in $\BBP_\BBC^2$. 
\begin{theorem}\cite[Theorem 1]{HH21}\label{thm: classification supersolvable}
    Let $\mathcal{L}$ be a line arrangement (over any field) with the number of modular points $\mu_{\mathcal{L}}>0$.
    \begin{enumerate}
        \item If $\mathcal{L}$ is not homogeneous, then either $\mathcal{L}$ is a near pencil or $\mu_{\mathcal{L}}=2$; if $\mu_{\mathcal{L}}=2$,
then $\mathcal{L}$ consists of $a\geq 2$ lines through one modular point, $b(>a)$ lines through the
other modular point, and we have $s_{\mathcal{L}}=a+b-1$ and $t_2=(a-1)(b-1)$.
\item If $\mathcal{L}$ has a modular point of multiplicity $2$, then $\mathcal{L}$ is trivial.
\item If $\mathcal{L}$ is complex and homogeneous with $m=m_{\mathcal{L}}>2$, then $1\leq \mu_{\mathcal{L}} \leq 4$. If $3\leq \mu_{\mathcal{L}} \leq 4$, we have the following possibilities. If $\mu_{\mathcal{L}}=4$, then $s_{\mathcal{L}}=6$, $m=3$, $t_2=3$, $t_3=4$ and $t_k=0$ otherwise; up to the change of coordinates, $\mathcal{L}$ consists of the lines $x=0$, $y=0$, $z=0$, $x-y=0$, $x-z=0$ and $y-z=0$ (intuitively, an equilateral triangle and its angle bisectors). And if $\mu_{\mathcal{L}}=3$, then $m>3$, and up to the change of coordinates, $\mathcal{L}$ consists of the lines defined by the linear factors
of $xyz(x^{m-2}-y^{m-2})(x^{m-2}-z^{m-2})(y^{m-2}-z^{m-2})$; hence, $s_{\mathcal{L}}=3(m-1)$, $t_2=3(m-2)$, $t_3=(m-2)^2$, $t_m=3$ and $t_k=0$ otherwise.

    \end{enumerate}
\end{theorem}

\vskip 2mm

If $\mathcal{L}$ is not homogeneous. the cases are illustrated in \ref{exam: near pencil} and \ref{exam: not having cycle of length 10}. Now assume that $\mathcal{L}$ is complex and homogeneous and $\mu_{\mathcal{L}} =4$.

\begin{theorem}
    Let $\mathcal{L}$ be a supersolvable line arrangement of $s_{\mathcal{L}}$ lines with $\mu_{\mathcal{L}} =4$. Then length of the longest induced cycle in the associated Levi graph is $8$.
\end{theorem}

\begin{proof}
 From theorem \ref{thm: classification supersolvable} we know that $s_{\mathcal{L}} = 6$ and the lines are given by:

 \begin{align*}
     L_x : x=0, L_y: y=0, L_z: z=0, \\
     L_{xy} = x-y =0, L_{xz}: x-z =0, L_{yz}: y-z =0.
 \end{align*}
The intersection points are $p_1 = (1,0,0), p_2=(0,1,0), p_3 =(0,0,1), p_4 =(1,1,0), p_5 =(1,0,1), p_6 = (0,1,1), p_7 =(1,1,1)$. All the triple intersection points i.e $p_1,p_2, p_3$ and $p_7$ are modular points. 

Note that, $L_{xy}p_4L_zp_1L_{yz}p_6L_xp_3L_{xy}$ corresponds to an induced cycle of length $8$ in the associated Levi graph.

Now assume that $\ell_1p_{12}\ell_2 \dots p_{51}\ell_1$ corresponds to an induced cycle pf length $10$. Let $S= \{p_{ii+1}, \, i =1, \dots, 4, p_{51} \}$. We claim that $S$ has at most one double points. Suppose that $S$ has two or three double points. Wlog assume that $p_{12} =p_4$ is a double point. Then $\ell_1 = L_{xy}$ and $\ell_2= L_z$. Choices for $\ell_3$ are $L_{xz}$ and $l_{yz}$. If $\ell_3 = L_{xz}$ then we choose $\ell_4 = L_y$ so that $p_{34}= p_5$ is a double point. Since $p_{45}$ has to be a triple point, choices are $(1,0,0)$ and $(0,0,1)$ But $L_z$ passe through $(1,0,0)$ and $L_{xy}$ passes through $(0,0,1)$. So, $S$ has at most on double points. But in that case it is not possible to obtain an induced $\BFC_{10}$ as $s_{\mathcal{L}}=6$ and  we need at least $4$ triple points. Hence, length of a longest induced cycle is $8$.
\end{proof}

We now discuss the cases when $\mu_{\CLL}=2,3$.

\vskip 2mm

$\mathbf{\mu_{\CLL}=3}$:

\begin{theorem}
    Let $\CLL$ be a supersolvable line arrangement of $s_{\CLL}$ lines in $\BBP_\BBC^2$ with $\mu_{\CLL}=3$. Then the associated Levi graph $G$ contain induced cycles $\BFC_{2i}$ for $$ 4\leq i \leq 2m-2$$.
\end{theorem}
\begin{proof}
    Let $\CLL$ be a supersolvable line arrangement of $s_{\CLL}$ lines in $\BBP_\BBC^2$ with $\mu_{\CLL}=3$. Then $m=m_{\CLL}>3$ and
up to change of coordinates, $\CLL$ consists of the lines defined by the linear factors of 
$$xyz(x^{m-2}-y^{m-2})(x^{m-2}-z^{m-2}) (y^{m-2}-z^{m-2})$$
such that $s_{\CLL}=3(m-1)=s$ (say), $t_2=3(m-2)$, $t_3=(m-2)^2$, $t_m=3$ and $t_k=0$ otherwise. Let $\ep$ be the $(m-2)$-th root of unity, that is, $\ep^{m-2}=1$. We fix the following notations for the lines:
\begin{equation}\label{notations lines supersolvable}
\begin{split}
    L_{xy}^i: x-\ep^iy \text{ for } 0\leq i\leq m-3; \\
     L_{xz}^i: x-\ep^iz \text{ for } 0\leq i\leq m-3; \\
    L_{yz}^i: y-\ep^iz \text{ for } 0\leq i\leq m-3; \\
    L_x: x=0, L_y: y=0 \text{ and } L_z: z=0. \\
    \end{split}
\end{equation}
Then double intersection points can be written in the form $ \{(\epsilon^i, 1, 0), (0, \epsilon^i,1), (1,0, \epsilon^i): i=0, \dots , m-3\}$, triple intersection points can be written in the form $\{(\epsilon^i, \epsilon^j, 1): i,j  \in \{0, \dots, m-3\} \}$ and three $m$-fold points are $(1,0,0), (0,1,0), (0,0,1)$.

Let $G$ be the associated Levi graph of $\CLL$. First, we show that for $0\leq i\leq m-3$, $G$ contains an induced cycle of length $2(2i+4)$.
Let us consider the lines for $0\leq i\leq m-3$: $\ell_1=L_z,\ell_2=L_y,\ell_3=L^0_{xz},\ell_4=L^0_{xy},\ell_5=L^1_{xz},\ell_6=L^1_{xy},\cdots, \ell_{2i+3}=L_{xz}^{i},\ell_{2i+4}=L_{xy}^{i}$ and intersection points accordingly to make an induced cycle $\BFC_{2(2i+4)}$ of length $2(2i+4)$ (See Figure \ref{fig: levi graph for supersolvable line arrangements}). The intersection points of $x-\epsilon^jz=0,x-\epsilon^jy=0$ and $x-\epsilon^jy=0,x-\epsilon^{j+1}z=0$ are $(\epsilon^j,1,1)$ and $(\epsilon^{j+1},\epsilon,1)$ respectively for $0\leq j\leq i$. Note that the set of lines, $\{\ell_1,\dots,\ell_{2i+4}\}\subseteq V(\BFC_{2(2i+4)}).$ Clearly, except the lines $x-\epsilon^jz=0,x-\epsilon^jy=0$, the point $(\epsilon^j,1,1)$ does not lie on any other line of the cycle $\BFC_{2(2i+4)}$ i.e., $(\epsilon^j,1,1)\notin \{\ell_1,\dots,\ell_{2i+4}\}\setminus \{L^j_{xy},L^j_{xz}\}$. The same holds for the point $(\epsilon^{j+1},\epsilon,1)$ too. Therefore, $\BFC_{2(2i+4)}$ is an induced cycle of length $2(2i+4)$ for $0\leq i\leq m-3$.
			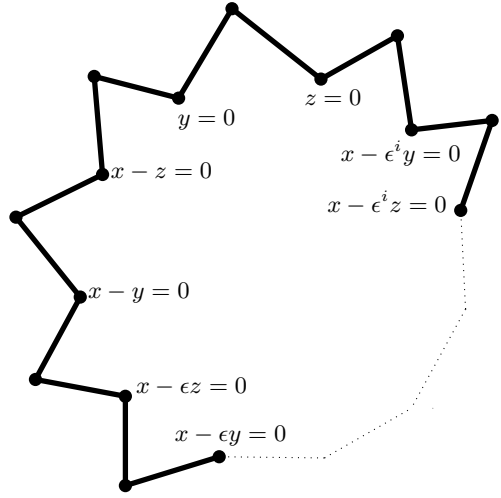
\begin{figure}[H]
				\begin{tikzpicture}[scale=1]
    \draw [line width=2pt] (-6.034651315755617,1.8029131271093046)-- (-4.883286264880086,2.3707799814009505);
\draw [line width=2pt] (-6.034651315755617,1.8029131271093046)-- (-5.18,0.74);
\draw [line width=2pt] (-5.18,0.74)-- (-5.775413087363905,-0.3564659578188849);
\draw [line width=2pt] (-5.775413087363905,-0.3564659578188849)-- (-4.580154887565866,-0.5801328168713076);
\draw [line width=2pt] (-4.580154887565866,-0.5801328168713076)-- (-4.579630177676801,-1.768644034773291);
\draw [line width=2pt] (-4.579630177676801,-1.768644034773291)-- (-3.332965668433068,-1.3844791437428416);
\draw [line width=2pt] (-4.883286264880086,2.3707799814009505)-- (-4.994413929413232,3.675382933071262);
\draw [line width=2pt] (-4.994413929413232,3.675382933071262)-- (-3.8724287561210886,3.3864272043145647);
\draw [line width=2pt] (-3.8724287561210886,3.3864272043145647)-- (-3.1626675139868152,4.576386157347569);
\draw [line width=2pt] (-3.1626675139868152,4.576386157347569)-- (-1.98,3.64);
\draw [line width=2pt] (-1.98,3.64)-- (-0.9633232103380183,4.215882602131701);
\draw [line width=2pt] (-0.9633232103380183,4.215882602131701)-- (-0.7751850559511861,2.964409855262814);
\draw [line width=2pt] (-0.7751850559511861,2.964409855262814)-- (0.29543528182319223,3.089683865185845);
\draw [line width=2pt] (0.29543528182319223,3.089683865185845)-- (-0.12186712229154395,1.8931363358292463);
\draw [dotted] (-0.05786847274836271,0.579552856470155) -- (-0.779083029050043,-0.7406158557312259);
\draw [dotted] (-1.9344907301099419,-1.4006693023343526) -- (-0.779083029050043,-0.7406158557312259);
\draw [dotted] (-0.05786847274836271,0.579552856470155) -- (-0.12186712229154395,1.8931363358292463);
\draw [dotted] (-1.9344907301099419,-1.4006693023343526) -- (-3.332965668433068,-1.3844791437428416);
\begin{scriptsize}
\fill  (-1.9344907301099419,-1.4006693023343526) circle ;
\fill  (-0.05786847274836271,0.579552856470155) circle ;
\fill  (-0.5,-0.7406158557312259) circle ;
\fill  (-1.98,3.64) circle (2.5pt);
\draw  (-1.82,3.4) node {$z=0$};
\fill  (-5.18,0.74) circle (2.5pt);
\draw  (-4.4,0.8) node {$x-y=0$};
\fill  (-3.8724287561210886,3.3864272043145647) circle (2.5pt);
\draw  (-3.5,3.1) node {$y=0$};
\fill  (-4.883286264880086,2.3707799814009505) circle (2.5pt);
\fill (-4.1,2.41) node {$x-z=0$};
\fill  (-4.580154887565866,-0.5801328168713076) circle (2.5pt);
\draw (-3.7,-0.45) node {$x-\epsilon z=0$};
\fill  (-3.332965668433068,-1.3844791437428416) circle (2.5pt);
\draw (-3.18,-1.1) node {$x-\epsilon y=0$};
\fill  (-0.7751850559511861,2.964409855262814) circle (2.5pt);
\draw (-0.92,2.65) node {$x-\epsilon^iy=0$};
\fill  (-0.12186712229154395,1.8931363358292463) circle (2.5pt);
\draw (-1.1,2.0) node {$x-\epsilon^iz=0$};
\fill  (-4.994413929413232,3.675382933071262) circle (2.5pt);
\fill  (-3.1626675139868152,4.576386157347569) circle (2.5pt);
\fill  (-6.034651315755617,1.8029131271093046) circle (2.5pt);
\fill  (-5.775413087363905,-0.3564659578188849) circle (2.5pt);
\fill  (-4.579630177676801,-1.768644034773291) circle (2.5pt);
\fill  (-0.9633232103380183,4.215882602131701) circle (2.5pt);
\fill  (0.29543528182319223,3.089683865185845) circle (2.5pt);
\end{scriptsize}

    \end{tikzpicture}
				\caption{Levi Graph $G=C_{2(2i+4)}$ associated to the supersolvable line arrangement $\mathcal{L}$}\label{fig: levi graph for supersolvable line arrangements}
			\end{figure}

We show that $G$ contains an induced cycle of length $2(2i+3)$ for every $1\leq i\leq m-3$. For $i=1$, we consider the lines $\ell_1=L^0_{xy}$, $\ell_2=L_z$, $\ell_3=L^1_{xy}$, $\ell_4=L^0_{xz}$, $\ell_5=L^1_{yz}$. In addition, we consider the corresponding intersection points $p_{12}=(1,1,0)$, $p_{23}=(\epsilon,1,0)$, $p_{34}=(\epsilon,1,\epsilon)$, $p_{4,5}=(1,\epsilon,1)$ and $p_{15}=(\epsilon,\epsilon,1)$. Clearly, $p_{ii+1}\notin L_{j}$ for $j\in [1,\dots,5]\setminus \{i,i+1\}$ and $p_{15}\notin L_j$ for $[1,\dots,5]\setminus \{1,5\}$. Therefore, the lines $\ell_1,\dots,\ell_5$ with those intersection points correspond to an induced cycle $C_{2\cdot 5}$ of length $10$. 
We now show that there exists an induced cycle of length $2(2i+3)$ for $2\leq i\leq m-5$. For that,
we consider the lines $\ell_1=L^0_{xy}$, $\ell_2=L_z$, $\ell_3=L^1_{xy}$, $\ell_4=L^0_{xz}$, $\ell_5=L_y$, $\ell_6=L^1_{xz}$, $\ell_7=L^2_{yz}$,\dots, $\ell_{2i+2}=L^{i-1}_{xz}$, $\ell_{2i+3}=L^i_{yz}$ if $i$ is even and

$\ell_1=L^0_{xy}$, $\ell_2=L_z$, $\ell_3=L^1_{xy}$, $\ell_4=L^0_{xz}$, $\ell_5=L_y$, $\ell_6=L^1_{xz}$, $\ell_7=L^2_{yz}$,\dots, $\ell_{2i+2}=L^{i+2}_{xz}$, $\ell_{2i+3}=L^{i+2}_{yz}$ if $i$ is odd.
Note that if $m-5$ is even, then $\ell_{2(m-5)+2}=L^{m-6}_{xz},$ $\ell_{2(m-5)+3}=L^{m-5}_{yz}$ and if $m-5$ is odd, then $\ell_{2(m-5)+2}=L^{m-4}_{xz},$ $\ell_{2(m-5)+3}=L^{m-3}_{yz}$. We also consider the corresponding intersection points such that it gives an induced cycle of length $2(2i+3)$ for $1\leq i\leq m-5$.
For $i=m-4$ and $i=m-3$, it will heavily depend on even, odd properties.
\\ \noindent
\textbf{Case-I:} We assume that $m$ is even. Then $m-4$ is even and consider the above lines for $m-5$ being odd with $\ell_{2(m-4)+2}=L^{m-5}_{xz}$ and $\ell_{2(m-4)+3}=L^{m-4}_{yz}$. So, we are left to make an induced cycle of length $2(2i+3)$ for $i=m-3.$ We just add the line $L_x$ to the above set of lines for $m-5$ being odd and consider the following:
$\ell_1=L^0_{xy}$, $\ell_2=L_z$, $\ell_3=L^1_{xy}$, $\ell_4=L^0_{xz}$, $\ell_5=L_y$, $\ell_6=L^1_{xz}$, $\ell_7=L^2_{yz}$,\dots, $\ell_{2(m-6)+2}=L^{m-7}_{xz}$, $\ell_{2(m-6)+3}=L^{m-6}_{yz}$, $\ell_{2(m-5)+2}=L^{m-4}_{xz}$, $\ell_{2(m-5)+3}=L^{m-3}_{yz}$, $\ell_{2(m-4)+2}=L_x$, $\ell_{2(m-4)+3}=L^{m-4}_{yz}$, $\ell_{2(m-3)+2}=L^{m-6}_{xz}$, $\ell_{2(m-3)+3}=L^{m-5}_{yz}$.
\\ \noindent
\textbf{Case-II:} We assume that $m$ is odd. Then $m-5$ is even. For $i=m-4$, we consider the set of lines for $m-5$ and add two lines $\ell_{2(m-4)+2}=L_x$ and $\ell_{2(m-4)+3}=L^{m-3}_{yz}$. For $i=m-3$, we consider the set of lines for $m-5$ and add four lines $\ell_{2(m-4)+2}=L_x$, $\ell_{2(m-4)+3}=L^{m-6}_{yz}$, $\ell_{2(m-3)+2}=L^{m-4}_{xz}$ and $\ell_{2(m-3)+3}=L^{m-3}_{yz}$. Taking the corresponding intersection points, we can make an induced cycle $C_{2(2i+3)}$ for $2\leq i\leq m-3$.

We now show that there is no induced cycle of length $>2(2m-2)$ in the Levi graph $G$.
Suppose that there is an induced cycle of length $r>2(2m-2).$ Therefore, the number of lines must be $>2m-2$. If there is one $m-$ fold intersection point, then we should avoid the $m-2$ lines containing it. For example, if we take the lines $L_y$, $L_z$ and the corresponding $m$-fold intersection point $(1,0,0)$. Then we must avoid the lines $L^{i}_{yz}$ for $0\leq i\leq m-3$. In this case, if we can take all the other possible lines except $L_x$, then we can also take the number of lines $\leq 2m-2$. So, we can assume that all the intersection points are either double intersection points or triple intersection points.

\end{proof}
 \noindent
 $\mathbf{\mu_{\CLL}=2:}$

 \vskip 2mm
 
Let $\mathcal{A}$ be an $m$-homogeneous supersolvable line arrangement with $m\geq 3$, having $\mu_{\mathcal{A}}=2$ modular point. We define the line arrangement $\mathcal{A}(w,k)$ 
\begin{align}
\mathcal{A}(w,k) \ \ : \ \ \ xyz(x^{m-2}-y^{m-2})(x^{m-2}-z^{m-2})\prod_{j=1}^{k} (y-\epsilon^{i_j}z),
\end{align}
where $0\leq i_j\leq m-3$ and $\epsilon^{m-2}=1$. For $k=0$, we define $\mathcal{A}(w,0)$ to be $xyz(x^{m-2}-y^{m-2})(x^{m-2}-z^{m-2})$.

Then it follows from \cite[Theorem 1.3]{AD20} that there is a unique $k$ with $0\leq k\leq m-3$ such that $\widetilde{\mathcal{A}}$ and $\widetilde{\mathcal{A}(w,k)}$ are lattice-isotopic, where $\widetilde{\mathcal{A}}$ and $\widetilde{\mathcal{A}(w,k)}$ are corresponding central plane arrangements in $\mathbb{C}^3$. In particular, intersections lattices $L(\widetilde{\mathcal{A}})$ and $L(\widetilde{\mathcal{A}(w,k)})$ are isomorphic i.e. both $\widetilde{\mathcal{A}}$ and $\widetilde{\mathcal{A}(w,k)}$ have the same combinatorics.

 \begin{theorem}\label{thm: supersolvable k < m-2}
 Let $\mathcal{A}(w, k)$ be the line arrangements as described above. Then
 
 $(i)$ For $k=0$ the associated Levi graph of $\mathcal{A}(w, 0)$ contains induced cycles of length $2i$ for all $i \in \{2,4, \dots, 2m-2\} \cup \{3\}$ and the maximum length of an induced cycle is $2(2m-2)$.

 $(ii)$ For $k=1$ the associated Levi graph of $\mathcal{A}(w, 1)$ contains induced cycles of length $2i$ for all $ i \in \{2, \dots, 2m-2\}$ and the maximum length of an induced cycle is $2(2m-2)$.

 $(iii)$ For $ 2 \leq k \leq m-4$ the associated Levi graph of $\mathcal{A}(w, k)$ has induced cycles of length $2i$ for all $i \leq 2m$. Moreover, for $k=2,3,4$, the maximum length of an induced cycle is $4m$.

 \end{theorem}

\begin{proof}
   Let us denote the associated Levi graphs of $\mathcal{A}(w, k)$ as $G_k$ and let $\ell_1p_{12}\ell_2 \dots \ell_{i-1}p_{{i-1}i}\ell_i$ corresponds to an induced cycle of length $2i$ in $G_k$ where $p_{ij} \in \sing(\mathcal{L}) \cap \ell_i \cap \ell_j $. We give the lines of $\mathcal{A}(w, k)$, the same notations as in (\ref{notations lines supersolvable}).
   
   $\textit{Proof of $(i)$}:$ The number of lines in $\mathcal{A}(w, 0)$ is $k= 2m-1$. We consider the lines: $\ell_1= L_{xy}^0, \ell_2=L_z, \ell_3= L_{xy}^1, \ell_4= L_{xz}^1, \dots, \ell_{2j+1} = L_{xy}^j, \ell_{2j+2} = L_{xz}^j$ for $1 \leq j \leq m-3$. Then $\ell_1p_{12}\ell_2 \dots \ell_{2i+2}p_{{2i+2}1}\ell_1$ corresponds to an induced cycle of length $2i$ for $i = 2j+2 =  2,3, 4, 6, \dots, 2m-4$.

   Next consider the lines: $\ell_{2m-3}= L_y, \ell_{2m-2} = L_{xz}^0$. Then $\ell_1p_{12}\ell_2 \dots \ell_{2m-2}p_{{2m-2}1}p_1$ corresponds to an induced cycle of length $2(2m-2)$. 

  Now, we claim that the maximum length of an induced cycle in $G_0$ is $2(m-2)$. If not, suppose there is an induced cycle of length $2(2m-1)$. Then $\ell_{i_1} = L_x$ for some $i_1\in \{1, \dots 2m-1\}$. In that case, there would be four choices for $\ell_{{i_1}+1}$, namely $L_{xy}^{j_1}, L_{xz}^{j_2}, L_y$ and $L_z$ for some $j_1, j_2 \in \{0, \dots m-3\}$. If $\ell_{{i_1}+1} = L_{xy}^{j_1}$, we have $p_{{i_1}{{i_1}+1}} = (0,0,1)$ which implies that $\{L_{xy}^j, j \neq j_1\} \cap \{ \ell_i, i \neq i_1+1\} = \emptyset$. Hence, we get an induced cycle of length much lower that $2(2m-1)$. Similar conclusions can be made when $\ell_{{i_1}+1} = L_{xz}^{j_2}, L_y$ or $L_z$. So, there is no induced cycle of length $2(m-1)$ in $G_0$.

  Next, we show that there are no induced cycles of length $2i$ in $G_0$ if $i$ is odd and $i \neq 3$. First consider the following two sets:
  \begin{align*}
      S_1:= \{L_z; L_{xy}^j, \, \, 0 \leq j \leq m-3 \} \\
      S_2:= \{L_y; L_{xz}^j, \, \, 0 \leq j \leq m-3 \}.
  \end{align*}
  Intersection point of any two lines from $S_1$ is $(0,0,1)$. Let $\ell_{v_1} = L_{xy}^{j_1}$ and $\ell_{v_2} = L_{xy}^{j_2}$ for some $v_1, v_2 \in \{1, \dots i\}$ and for some $j_1, j_2 \in \{0, \dots , m-3\}$. Then $p_{v_1{{v_1}+1}} = (0,0,1)$. Also, observe that $\{\ell_i, i \neq v_1, v_{i+1} \} \cap S_2 = \emptyset$ and $\{\ell_i, i \neq v_1, v_{i+1} \} \cap S_1 \setminus \{L_{xy}^{j_1}, L_{xz}^{j_2}\} = \emptyset$. so,$\{j_1, j_2\} = \{1,2\}$ and the length of the induced cycle is $4$. Similar conclusions can be made if $\ell_{v_1} = L_z, \ell_{v_2} = L_{xy}^{j_1}$ for some $j_1 \in \{0, \dots, m-3\}$ or $\ell_{v_1}, \ell_{v_2} \in S_2$. 
  
  Hence, if $\ell_{v_1} \in S_1$, $\ell_{{v_1}+1} \in S_2$ for any $v_1 \in \{1, \dots, i\}$. Therefore, $G_0$ does not contain any induced cycle of length $2i$ when $i$ is odd and not equal to three.
  
\vskip 2mm

  \textit{Proof of $(ii)$}: For $k=1$ we choose $i_1=1$ in $(16)$ so that 
  \begin{align*}
  \mathcal{A}(x,1)= xyz(x^{m-2}-y^{m-2})(x^{m-2}-z^{m-2})(y - \epsilon z).
  \end{align*}
   The cases where $i_1 \neq 1$ can be proved similarly. We consider the lines: $\ell_1= L_y, \ell_2 = L_{yz}^1, \ell_3 = L_{xy}^0, \ell_4 = L_{xz}^0, \ell_5 = L_{xy}^1, \ell_6 = L_z, \ell_7 = L_{xy}^2, \ell_8 = L_{xz}^2, \dots, \ell_{2j+3} = L_{xy}^j, \ell_{2j+4} = L_{xz}^j$ for $2 \leq j \leq m-3$. Then $\ell_1p_{12}\ell_2 \dots \ell_{2j+4}p_{(2j+4)1}p_1$ corresponds to an induced cycle of length $2i$ for $i$ even and $ 8 \leq i = 2j+4 \leq 2m-2$. For $i=6$, $\ell_1p_{12}\ell_2 \dots, \ell_5p_{5(10)}\ell_{10}p_{(10)1}\ell_1$ corresponds to an induced cycle of length $12$ in $G_1$.
  
  Now to show that $G_1$ contains induced cycles of length $2i$ for $i$ odd, we consider the lines: $\ell_1 = L_{xy}^0, \ell_2 = L_{yz}^1, \ell_3 = L_{xz}^0, \ell_4 = L_{xy}^1, \ell_5 = L_z, \ell_6 = L_{xy}^2, \ell_7 = L_{xz}^2, \dots \ell_{2j+2} = L_{xz}^j, \ell_{2j+3} = L_{xy}^j$ for $ 2 \leq j \leq m-4$. Then $\ell_1p_{12}\ell_2 \dots \ell_{2j+3}p_{(2j+3)1}\ell_1$ corresponds to an induced cycle of length $2i$ for $i$ odd and $ 7 \leq i = 2j+3 \leq 2m-5$. For $i=5$ $\ell_1p_{12}\ell_2 \dots \ell_5p_{51}\ell_1$ corresponds to an induced cycle of length $10$ in $G_1$.

Next consider the lines: $\ell_{2m-4} = L_y, \ell_{2m-3} = L_{xz}^{m-3}$. Then,
\begin{align*}
\ell_1p_{12}\ell_2 \dots \ell_{2m-5}p_{(2m-5)(2m-4)}\ell_{2m-4}p_{(2m-4)(2m-3)}\ell_{2m-3}p_{(2m-3)1}p_1\ell_1
\end{align*}
corresponds to an induced cycle of length $2(2m-3)$.

Hence, $G_1$ contains induced cycles of length $2i$ for all $i \leq 2m-2$.

Now we show that maximum length of an induced cycle in $G_1$ is $2(2m-2)$. If $\ell_i \neq L_{yz}^1$ for $i \in \{1, \dots, m-3\}$, then by $(i)$ we see that the maximum length of an induced cycle in $G_1$ is $2(m-2)$. So, let $\ell_{v_1} = L_{yz}^1$ for some $v_1 \in \{1, \dots, m-3\}$. If $\ell_{v_1 -1}, \ell_{v_1+1} \in S_1$, then there exists at least one $j_1$ such that $\{ \ell_i, i \neq i_1-1, i_1, i_1+1\} \cap L_{xz}^{j_1} = \emptyset$. If $\ell_{v_1-1}\ell_{v_1+1} \in S_2$ or $\ell_{v_1-1} \in S_1, \ell_{v_1+1} \in S_2$, conclusions are same.
Also, $\{ \ell_v, \, \, 1 \leq v \leq i \}\cap L_x = \emptyset$ for the same reason as described in $(i)$.

Hence, maximum length of an induced cycle in $G_1$ is $2(2m-2)$. 

\vskip 2mm 
\textit{Proof of $(iii)$}: For $k=2$ we choose $i_1 =1$ and $i_2=2$ so that 
\begin{align*}
  \mathcal{A}(x,2)= xyz(x^{m-2}-y^{m-2})(x^{m-2}-z^{m-2})(y - \epsilon z)(y - \epsilon^2z).
  \end{align*}

  We consider the lines: $\ell_1 = L_{xy}^0, \ell_2= L_{yz}^2, \ell_3 = L_x, \ell_4 = L_{yz}^1, \ell_5 = L_{xy}^1, \ell_6 = L_{xz}^1, \ell_7 = L_{xy}^2, \ell_8 = L_z, \ell_9 = L_{xy}^3, \ell_{10} = L_{xz}^3, \dots , \ell_{2j+3} = L_{xy}^j, \ell_{2j+4} = L_{xz}^j$ for $3 \leq j \leq m-3$. Then $\ell_1p_{12}\ell_2 ... \ell_{2j+4} p_{{2j+4}1}\ell_1$ corresponds to an induced cycle of length $2i$ for $9 \leq i= 2j+4 \leq 2m-2$ and $i$ even. 
  
  To show that $G_2$ contains an induced cycle of length $4m$, consider the lines: $\ell_{2m-1} = L_y, \ell_{2m} = L_{xz}^0$. Then, $\ell_1p_{12}\ell_2 \dots \ell_{2m}p_{(2m)1}\ell_1$ corresponds to an induced cycle of length $4m$ in $G_2$.

  Note that, there exists $j_1, j_2 \in \{1, \dots, m-3\}$ such that 
  \begin{align*}
  p_{(2j+3)(2j+4)} = \sing(\mathcal{L}) \cap L_{xy}^j \cap L_{xz}^j \cap L_{yz}^{j_1}, \\ 
  p_{(2j+2)(2j+3)} = \sing(\mathcal{L}) \cap L_{xz}^{j-1} \cap L_{xy}^j \cap L_{yz}^{j_2}
  \end{align*}
  
  for all $3 \leq j \leq m-3$. Also, 
  \begin{align*}
  p_{56} = \sing(\mathcal{L}) \cap L_{xy}^1 \cap L_{xz}^1 \cap L_{yz}^{j_1},\\
  p_{(2m)1} = \sing(\mathcal{L}) \cap L_{xz}^0 \cap L_{xy}^0 \cap L_{yz}^0, \\
  p_{67} = \sing(\mathcal{L}) \cap L_{xz}^1 \cap L_{xy}^2 \cap L_{yz}^{j_2}.
  \end{align*}

  Hence, we can conclude that $\ell_1p_{12}\ell_2 ... \ell_i p_{i1}\ell_1$ corresponds to an induced cycle of length $2i$ in $G_k$ for $2 \leq i \leq 2m$ and $i$ even, for $2 \leq k \leq m-4$.

  Now from the proof $(ii)$ we get induced cycles of length $2i$ in $G_2$ for $i$ odd and $3 \leq i \leq 2m-3$.

Now, we show that maximum length of an induced cycle in $G_2$ is $4m$. If $\{ L_{yz}^1, L_{yz}^2 \} \not \subset \{\ell_v : 1 \leq v \leq i \}$, we can go back to $(i)$ and $(ii)$ and conclude that maximum length of an induced cycle in $G_2$ is strictly less than $4m$. Suppose there exists a $v_1 \in \{1, \dots , i\}$ such that $\ell_{v_1} = L_{yz}^1$. If $\ell_{{v_1}+1} = L_x$ and  $\ell_{{v_1}+2} = L_{yz}^2$, then $\ell_{{v_1}+3} = L_{xy}^{j_1}$ or $L_{xz}^{j_2}$ for some $j_1, j_2 \in \{1, \dots, m-3\}$. If $\ell_{{v_1}+3} = L_{xy}^{j_1}$ then there exists a $j_3 \in \{1, \dots, m-3\}$ such that $p_{({v_1}+2)({v_1}+3)} \in \sing(\mathcal{L}) \cap L_{yz}^2 \cap L_{xy}^{j_1} \cap L_{xz}^{j_3}$. Hence, $\{ \ell_v: 1 \leq v \leq i\} \cap L_{xz}^{j_3} = \emptyset$ and maximum length of an induced cycle is $\leq 4m$. But we have already shown the existence of an induced cycle of length $4m$. So, maximum length of an induced cycle in $G_2$ is $4m$. Now, if  $\ell_{{v_1}+3} = L_{xz}^{j_2}$ or $\ell_{{v_1}+1} \neq  L_x$, we can make similar conclusions. Finally observe that for $(i_1, i_2) \neq (1,2)$ we will have a similar proof and same conclusion.

For $k=3,4$, as before we assume that $(i_1, \dots, i_4) = (1, \dots, 4)$. The proof is similar when $(i_1, \dots, i_4) \neq  (1, \dots, 4)$. 

We claim that maximum length of an induced cycle in $G_3$ and $G_4$ is $4m$. First we prove the claim for $G_3$. If $ ( L_{yz}^1, \dots, L_{yz}^3) \not \subset \{\ell_v : 1 \leq v \leq i \}$ we can consider the induced cycles of $G_3$ as induced cycles in $G_0, G_1$ or $G_2$ and can conclude that maximum length of an induced cycle is $\leq 4m$. So, assume that $ ( L_{yz}^1, \dots, L_{yz}^3) \subset \{\ell_v : 1 \leq v \leq i \}$. Let, $\ell_{v_1} = L_{yz}^{i_1}, \ell_{{v_1}+1} = L_x$ and $\ell_{{v_1}+2} = L_{yz}^{i_2}$ for some $v_1 \in \{1, \dots, i\}$. Then $\ell_{{v_1}+3} = L_{x*}^{j_1}$ for $* \in \{y,z\}$ and for some $j_1 \in \{0, \dots, m-3\}$. So, there exists $j_2 \in \{0, \dots, m-3 \}$ such that $L_{x*}^{j_3} \cap \{ \ell_v : 1 \leq v \leq i \} = \emptyset$, where $* \in \{y,z\}$. Next let $\ell_{v_2} = L_{yz}^{i_3}$ for some $v_2 \in \{1, \dots, i\}$ with $v_2 \in \{ v_1, v_1 +1, \dots, v_1+3\}$. Then $\ell_{{v_2} -1}$ or $\ell_{{v_2}+1}$  is equal to  $L_{x*}^{j_3}$ for $* \in \{y,z\}$ and for some $j_3 \in \{0, \dots, m-3\}$ with $j_3 \neq j_1, j_2$. So, there exists $j_4 \in \{0, \dots, m-3\}$ such that $L_{x*}^{j_4} \cap \{ \ell_v : 1 \leq v \leq i \} = \emptyset$. Hence 
\begin{align*}\label{thm ex 2}
\{L_{x*}^{j_3}, L_{x*}^{j_4}\} \cap \{ \ell_v : 1 \leq v \leq i \} = \emptyset. 
\end{align*}

Now assume that $\ell_{{v_1}+1} \neq L_x$. In that case too, there exists at least two $j_3, j_4 \in \{0, \dots, m-3 \}$ such that $\{L_{x*}^{j_3}, L_{x*}^{j_4} \} \cap \{\ell_v : 1 \leq v \leq i \} = \emptyset$. Hence, maximum length of an induced cycle in $G_3$ is $\leq 2m$. As we have already shown an induced cycle of length $2m$ in $G_3$, maximum length of an induced cycle in $G_3$ is $2m$.

For $k=4$ similarly we can show that there exists at least three lines $L_{x*}^{j_1}, L_{x*}^{j_2}, L_{x*}^{j_3}$ such that $\{ L_{x*}^{j_1}, L_{x*}^{j_2}, L_{x*}^{j_3} \} \cap \{\ell_v : 1 \leq v \leq i \} = \emptyset$ for $j_1, \dots, j_3 \in \{0, \dots, m-3 \}$ and $* \in \{y,z\}$. Hence, maximum length of an induced cycle in $G_4$ is $2m$.

\end{proof}

\bibliographystyle{plain}
\bibliography{Reference}
\end{document}